\documentclass[11pt, oneside]{amsart}   	
\usepackage{geometry}                		
\usepackage{graphicx}				
\usepackage{amsaddr}								
\usepackage{amssymb}
\usepackage{color}
\usepackage{amsmath}
\usepackage{amsthm}
\usepackage{hyperref}
\usepackage{appendix}
\usepackage{esint}

\allowdisplaybreaks
\numberwithin{equation}{section}

\newcommand{\Z}{\mathbb{Z}}
\newcommand{\R}{\mathbb{R}}

\newcommand{\N}{\mathbb{N}}
\renewcommand{\S}{\mathbb{S}}

\newcommand{\Ric}{\operatorname{Ric}}
\newcommand{\Rm}{\operatorname{Rm}}

\newcommand{\tr}{\operatorname{tr}}
\newcommand{\supp}{\operatorname{Supp}}

\newcommand{\id}{\operatorname{id}}
\newcommand{\loc}{\text{loc}}
\newcommand{\op}{\text{op}}

\renewcommand{\L}{\mathcal{L}}

\renewcommand{\mod}{\operatorname{minmod}}

\newcommand{\dist}{\operatorname{dist}}

\newcommand{\Supp}{\operatorname{Supp}}

\theoremstyle{plain}
\newtheorem{theorem*}{Theorem}
\newtheorem{corollary*}{Corollary}
\newtheorem{question*}{Question}
\newtheorem{definition*}{Definition}
\newtheorem{claim*}{Claim}
\newtheorem{theorem}{Theorem}[section]
\newtheorem{lemma}[theorem]{Lemma}
\newtheorem{corollary}[theorem]{Corollary}
\newtheorem{claim}[theorem]{Claim}

\theoremstyle{definition}
\newtheorem{definition}[theorem]{Definition}
\theoremstyle{remark}
\newtheorem{remark}[theorem]{Remark}

\title{A continuous Positive Mass Theorem for perturbations of Euclidean space}
\author{Paula Burkhardt-Guim}
\address{Stony Brook University}
\email{paula.burkhardt-guim@stonybrook.edu}
\date{\today}

\begin{document}
\maketitle

\section{Introduction}

\begin{abstract}
We prove a Positive Mass Theorem for $C^0$-asymptotically flat Riemannian metrics with nonnegative scalar curvature in a weak sense that are sufficiently uniformly close to Euclidean space. More precisely, we show that a $C^0$-asymptotically flat Riemannian metric that is a $C^0$ perturbation of Euclidean space with nonnegative scalar curvature in the sense of Ricci flow has nonnegative mass, where the mass is given by a $C^0$ analog of the classical ADM mass previously introduced by the author \cite{PBG23}.
\end{abstract}

In recent years considerable evidence has emerged to suggest that lower scalar curvature bounds are a $C^0$ property of Riemannian metrics \cite{Gromov14}, \cite{Gromov21}, \cite{Bamler16}, \cite{PBG19}, \cite{PBG20}, \cite{Jauregui20-2}, \cite{Li20-2}, \cite{MazurowskiYao26-1}, \cite{FogagnoloGattiPluda26}. In light of this, it is natural to ask whether the Riemannian Positive Mass Theorem, a celebrated theorem about nonnegative scalar curvature, holds for Riemannian metrics that are merely continuous. The goal of this paper is to prove a result of this type. 

The classical Riemannian Positive Mass Theorem deals with smooth Riemannian metrics that are $C^2$-asymptotically flat with everywhere nonnegative scalar curvature. We use the term ``$C^2$-asymptotically flat'' to mean the following (in the literature this property is often referred to as simply ``asymptotically flat''):
\begin{definition}\label{def:C2AF}
Let $n\geq 3$ and let $M^n$ be a smooth manifold. A $C^2$ Riemannian metric $g$ on $M$ is said to be $C^2$-asymptotically flat if there exists a compact subset $K\subset M$ and a coordinate chart $\Phi: M\setminus K\to \R^n\setminus \overline{B(0,1)}$ for $M$ such that, for some $\tau > \tfrac{1}{2}(n-2)$, $c_0 >0$, and $\bar r >0$ we have
we have
\begin{equation}\label{eq:C0AFdecay}
|(\Phi_*g)_{ij} - \delta_{ij}|\big|_x \leq c_0|x|_{\delta}^{-\tau},
\end{equation}
and 
\begin{equation}\label{eq:CkAFdecay}
|\partial_k (\Phi_*g)_{ij}|\big|_x \leq c_0|x|_{\delta}^{-\tau - |k|} \text{ for } |k|= 1,2,
\end{equation}
for all $|x|_\delta \geq \bar r$, where $\delta$ denotes the Euclidean metric and $k$ is a multiindex. We say $\tau$ is the decay rate of $g$.
\end{definition}

Recall that if $(M, g)$ is a smooth, $C^2$-asymptotically flat Riemannian manifold, then the ADM mass of $g$, introduced in \cite{ArnowittDeserMisner61}, is given by
(\cite[p. 143]{Schoen89}):
\begin{equation}\label{eq:classicalADMmassdef}
m_{ADM}(g) := \lim_{r\to \infty}\frac{1}{4\pi (n-1)\omega_{n-1}}\int_{\S(r)} \sum_{i=1}^{n}(\partial_i g_{ij} - \partial_j g_{ii})\nu^jdS,
\end{equation}
where the coordinate expression in the integrand corresponds to the coordinates $\Phi$ given by Definition \ref{def:C2AF}, 
\begin{equation*}
\S(r) = \{x\in \R^n : (x^1)^2 + \cdots + (x^n)^2 = r^2\},
\end{equation*}
 $\nu$ denotes the outward unit normal to $\S(r)$ with respect to the Euclidean metric, $\omega_{n-1}$ denotes the Euclidean volume of the $(n-1)$-dimensional unit sphere, and $dS$ denotes the Euclidean surface measure on $\S(r)$. A priori it is not clear whether the limit (\ref{eq:classicalADMmassdef}) should always exist, or whether the limit depends on the choice of $\Phi$, but Bartnik \cite[Theorems $4.2$ and $4.3$]{Bartnik86} (see also \cite{Chrusciel88} for the asymptotically Minkowski case) showed that if $g$ is $C^2$-asymptotically flat and $R(g) \in L^1(M)$, then $m_{ADM}(g)$ does indeed exist, is finite, and is independent of choice of $\Phi$ satisfying Definition \ref{def:C2AF}.

The classical Riemannian Positive Mass theorem then says the following:
\begin{theorem}\label{thm:C2PMT}
For $n\geq 3$ let $(M^n, g)$ be a smooth, $C^2$-asymptotically flat Riemannian manifold such that $R(g)\in L^1(M)$. If $R(g)\geq 0$ then $m_{ADM}(g) \geq 0$, and $m_{ADM}(g) = 0$ if and only if $g$ is flat.
\end{theorem}

We remark that we have only stated Definition \ref{def:C2AF} and Theorem \ref{thm:C2PMT} for Riemannian manifolds with a single asymptotically flat end, but the result also holds for multiple ends (see, e.g., \cite{SchoenYau81}, \cite{Witten81}). There are a number of proofs of the Riemannian Positive Mass Theorem using different techniques, the earliest of which were due to Schoen--Yau \cite{SchoenYau79}, \cite{SchoenYau81} for $3 \leq n \leq 7$ and Witten \cite{Witten81} for spin manifolds in all dimensions. We also refer the reader to \cite{Lohkamp99}, \cite{Li18}, \cite{Bray11}, \cite{Lee19} for a more thorough discussion of various proof techniques. More recently the case of higher dimensional non-spin manifolds has been addressed \cite{SchoenYau19}, \cite{ChodoshMantoulidisSchulze26}, \cite{BiHaoHeShiZhu26}, \cite{BrendleWang26}.

As initially explained, the goal of this paper is to prove a $C^0$ version of Theorem \ref{thm:C2PMT}. With this goal in mind, it is natural to replace the condition that $g$ be a smooth, $C^2$-asymptotically flat Riemannian metric with the following:
\begin{definition}\label{def:C0AF}
Let $n\geq 3$ and let $M^n$ be a smooth manifold. A $C^0$ Riemannian metric $g$ on $M$ is said to be $C^0$-asymptotically flat if there exists a compact subset $K\subset M$ and a coordinate chart $\Phi$ $\Phi: M\setminus K\to \R^n\setminus \overline{B(0,1)}$ for $M$ such that, for some $\tau > \tfrac{1}{2}(n-2)$, $c_0 >0$, and $\bar r >0$ we have that $g$ satisfies (\ref{eq:C0AFdecay}) for all $|x|_{\delta} > \bar r$, but not necessarily (\ref{eq:CkAFdecay}). We say $\tau$ is the decay rate of $g$.
\end{definition}
Therefore, in order to formulate a version of Theorem \ref{thm:C2PMT} using only $C^0$ data, it remains to make sense of the ADM mass and the lower scalar curvature bound for $C^0$ Riemannian metrics (in our result we will omit the condition that the metric have scalar curvature in $L^1(M)$; we discuss the consequences of this later).

We now describe the setting of our theorem. We work in any dimension $n\geq 3$. To make sense of the ADM mass, we use previous work in which we proved the following \cite[Theorems 1.3 and 2.9]{PBG23}:
\begin{theorem}\label{thm:C0ADMmassexistence}
Let $M$ be a smooth manifold, and $g$ a continuous Riemannian metric on $M$. Suppose there exists a compact set $K$ such that $M\setminus K$ is diffeomorphic to $\R^n\setminus \overline{B(0,1)}$. For any smooth cutoff function $\varphi: \R \to \R^{\geq 0}$ with $\supp(\varphi)\subset\subset(.9, 1.1)$ and for any $r>0$, there exists a smooth family of functions $(\varphi^r)_{r>0}: \R\to \R$ such that $\varphi^r\xrightarrow[r\to\infty]{C^\infty} \varphi$, and there exists a quantity $M_{C^0}(g, \varphi^r, r)$, depending on only the $C^0$ data of $g$ in coordinates, for which the following is true:
\begin{enumerate}
\item\label{item:generalC0limitexistence} If $g$ is $C^0$-asymptotically flat with decay rate $\tau$ for some $\tau > (n-2)/2$ and $g$ has nonnegative scalar curvature in the sense of Ricci flow on $M\setminus K$, then the limit $M_{C^0}(g) := \lim_{r\to \infty}M_{C^0}(g, \varphi^r, r)$ exists, is either finite or $+\infty$, and is independent of choice of $\varphi$ and choice of $C^0$-asymptotically flat coordinate chart for $g$. 
\item If $g$ is $C^2$ and $m_{ADM}(g)$ exists, then $m_{ADM}(g) = \lim_{r\to\infty}M_{C^0}(g, \varphi^r, r)$.
\end{enumerate}

\end{theorem}
We describe the quantity $M_{C^0}(g)$ in more detail in Section \ref{sec:preliminaries}. In the setting of this paper, the statement that $g$ has nonnegative scalar curvature ``in the sense of Ricci flow'' means that there exists a smooth Ricci-DeTurck flow $(g_t)_{t>0}$ starting from $g$ such that $R(g_t)\geq 0$ for all $t>0$, though one can also make sense of the condition in more general settings in which such a flow might not exist; we refer the reader to Section \ref{subsec:equivdefs} for the precise definition and some related properties. The Ricci-DeTurck flow is a geometric flow that is equal to a Ricci flow up to pullback by a family of diffeomorphisms. Because it is parabolic, one expects that the flow will be smooth for positive times even if the initial data is not smooth. The condition that $g$ has nonnegative scalar curvature in the sense of Ricci flow is a reasonable weak formulation of the condition that $g$ have nonnegative scalar curvature, since for $C^2$ initial data, uniform lower scalar curvature bounds are preserved by the Ricci-DeTurck (and Ricci) flow. For a more thorough discussion of this condition we refer the reader to \cite{PBG19}.

The main theorem of this paper is the following:
\begin{theorem}\label{thm:C0PMT}
Let $n\geq 3$ and $\tau \in (\tfrac{n-2}{2}, n-2)$. There exists $\bar \varepsilon = \bar \varepsilon(n, \tau)$ such that the following is true:
Suppose $g$ is a $C^0$-asymptotically flat Riemannian metric on $\R^n$ with decay rate $\tau$, in the sense that for some $c_0 >0$ we have
\begin{equation}\label{eq:C0AFisolated}
|g_{ij} - \delta_{ij}|(x) \leq c_0 |x|^{-\tau}
\end{equation}
for all $x\in \R^n\setminus \overline{B(0,1)}$. Also assume that $|| g- \delta||_{C^0(\R^n)}< \bar \varepsilon$, and suppose that $g$ has nonnegative scalar curvature in the sense of Ricci flow. Then $M_{C^0}(g)\geq 0$.
\end{theorem}
\begin{remark}
The condition that $|| g- \delta||_{C^0(\R^n)}< \bar \varepsilon$ in Theorem \ref{thm:C0PMT} implies the existence of a \emph{long-time} solution to the Ricci-DeTurck flow starting from $g$ that satisfies certain estimates (we will describe these estimates in more detail in Section \ref{sec:preliminaries}). That the solution exists for all $t>0$ is essential to the techniques in this paper. The dependence of $\varepsilon$ on $\tau$ is used to show that the Ricci-DeTurck flow will remain $C^0$-asymptotically flat with decay rate $\tau$ for all time, and that the constant $c_0$ in (\ref{eq:C0AFisolated}) does not degenerate as $t\to \infty$.
\end{remark}
\begin{remark}
Under the hypotheses of Theorem \ref{thm:C0PMT}, the quantity $M_{C^0}(g)$ can possibly be equal to $+\infty$, but it cannot be equal to $-\infty$, and when it is finite, Theorem \ref{thm:C0PMT} says that it must be nonnegative. That we can have $M_{C^0}(g) = +\infty$ is a consequence of omitting any version of the condition that the metric have scalar curvature in $L^1(M)$ (which appears in Theorem \ref{thm:C2PMT}) in our result. In this paper we do not address the rigidity portion of the Positive Mass Theorem (that is, what happens when $M_{C^0}(g) = 0$). For rigidity statements assuming faster decay, we refer the reader to \cite{MazurowskiYao26-1}, \cite{MazurowskiYao26-2}, \cite{ChuLeeWan26}.
\end{remark}

Having stated our main result, we note that in $3$ dimensions, multiple $C^0$ versions of the Positive Mass Theorem have already been proven using different techniques. The 3-dimensional problem was first studied by Huisken \cite{Huisken06}, who used the isoperimetric deficit of large sets to make sense of the ADM mass (see also \cite{JaureguiLeeUnger24}). Subsequently, Antonelli--Fogagnolo--Nardulli--Pozzetta \cite{AntonelliFogagnoloNardulliPozzetta26} used weak inverse mean curvature flow to prove nonnegativity of a quasi-local isoperimetric mass for $C^0$ Riemannian 3-manifolds that are $C^0_{loc}$-asymptotic to $\R^3$ \cite[Definition 1.2, Definition 1.3]{AntonelliFogagnoloNardulliPozzetta26}, provided that the $C^0$ Riemannian metric has nonnegative scalar curvature in the approximate sense (that is, it admits a locally uniform approximation by smooth metrics with almost nonnegative scalar curvature). Also in 3 dimensions, Mazurowski--Yao \cite{MazurowskiYao26-2} used harmonic functions to prove a version of Theorem \ref{thm:C2PMT} (with a rigidity statement) for $C^0$-asymptotically flat Riemannian metrics on $\R^3$ with nonnegative scalar curvature in the approximate sense, where the ADM mass is replaced by the ``harmonic mass''. The harmonic mass of $g$ is given by $\lim_{r\to\infty} M_{C^0}(g, \eta, r)$, where $M_{C^0}(g, \cdot, r)$ is the quantity from Theorem \ref{thm:C0ADMmassexistence}, with $\varphi^r$ replaced by the particular function $\eta$ of \cite[Theorem 3]{MazurowskiYao26-2}. They also show \cite[Remark 4]{MazurowskiYao26-2} that under the faster decay rate $\tau > 2/3$ the harmonic mass agrees with the quantity $M_{C^0}(g)$ from Theorem \ref{thm:C0ADMmassexistence}. The relation of the harmonic mass of \cite{MazurowskiYao26-2} to Huisken's isoperimetric mass is addressed in \cite{BenattiFogagnolo26}. The main novelty of our Theorem \ref{thm:C0PMT} is that it holds in arbitrary dimensions $n\geq 3$, and that the proof relies on different techniques than \cite{AntonelliFogagnoloNardulliPozzetta26} and \cite{MazurowskiYao26-2}.

Before we continue, we remark that there have been a number of versions of the Positive Mass Theorem for singular metrics with different regularity than $C^0$ (e.g. \cite{Miao02}, \cite{McFeronSzekelyhidi12}, \cite{LeeLeFloch15}, \cite{ChuLeeZhu22}, \cite{DaiSunWang25}, \cite{Hafemann26}, \cite{LeeLitzingerSimon26}; see also \cite{JiangShengZhang22} among others). In many of these (e.g. \cite{Miao02}, \cite{McFeronSzekelyhidi12}, \cite{ChuLeeZhu22}, \cite{JiangShengZhang22}, \cite{DaiSunWang25}, \cite{Hafemann26}, \cite{LeeLitzingerSimon26}), the metric is assumed to be $C^2$ and $C^2$-asymptotically flat away from some subset, allowing to make sense of the ADM mass and asymptotic flatness of the metric. Some of these (\cite{McFeronSzekelyhidi12}, \cite{ChuLeeZhu22}, \cite{LeeLitzingerSimon26}) use Ricci or Ricci-DeTurck flow to smooth the singular metric while preserving $C^2$-asymptotic flatness and nonnegative scalar curvature, to create a setting in which the classical Positive Mass Theorem may be applied (see also \cite{DaiMa07}, \cite{Appleton18}). In this paper, such an approach is not available to us because if $g$ is only $C^0$-asymptotically flat, one can only expect that the Ricci-DeTurck flow $g(t)$ for $g$ is $C^0$-asymptotically flat for $t>0$, but not necessarily $C^2$-asymptotically flat. In particular, one does not expect to be able to directly apply the classical Positive Mass Theorem to the time slices of the Ricci-DeTurck flow. This is the crux of the problem.

In order to deal with this, we apply a gluing and conformal deformation technique inspired by that of \cite{Schoen89} (see also \cite[Chapter 3]{Lee19}) to the time-slices of the flow. In order to offset the fact that $g(t)$ is only $C^0$-asymptotically flat, we take $t$ to be large when working on regions where $|x|$ is large. We then prove
\begin{theorem}\label{thm:masssequence}
Let $n\geq 3$ and $\tau \in (\tfrac{n-2}{2}, n-2)$. There exists $\bar \varepsilon = \bar \varepsilon(n, \tau)$ such that the following is true:
Suppose $g$ is a $C^0$-asymptotically flat Riemannian metric on $\R^n$ with decay rate $\tau$, in the sense that for some $c_0 >0$ we have
\begin{equation}\label{eq:C0AFisolated}
|g_{ij} - \delta_{ij}|(x) \leq c_0 |x|^{-\tau}
\end{equation}
for all $x\in \R^n\setminus \overline{B(0,1)}$. Also assume that $|| g- \delta||_{C^0(\R^n)}< \bar \varepsilon$, and suppose that $g$ has nonnegative scalar curvature in the sense of Ricci flow. Then there exists a family of Riemannian metrics $(g_t')_{t> 0}$ such that 
\begin{enumerate}
\item Each $g_t'$ is classically asymptotically flat (with scalar curvature in $L^1(\R^n)$),
\item Each $g_t'$ has nonnegative scalar curvature on $\R^n$, and
\item We have
\begin{equation}
M_{C^0}(g_t') \leq M_{C^0}(g) + \delta(t)
\end{equation}
for some $\delta(t)$ such that $\delta(t)\xrightarrow[t\to\infty]{} 0$.
\end{enumerate}
\end{theorem}
\begin{proof}[Proof of Theorem \ref{thm:C0PMT}, given Theorem \ref{thm:masssequence}]
Let $(g_t')$ be the family of metrics given by Theorem \ref{thm:masssequence}. We apply the classical Riemannian Positive Mass Theorem to each $g_t'$ to find that $M_{C^0}(g_t')\geq 0$ for all $t> t_0$. In particular, the inequality in Theorem \ref{thm:masssequence} implies that
\begin{equation}
-\delta(t) \leq M_{C^0}(g)
\end{equation}
for $t> 0$, so letting $t\to \infty$ we find that $M_{C^0}(g)\geq 0$.
\end{proof}

We now explain the structure of the rest of the paper. In Section \ref{sec:preliminaries} we describe the Ricci-DeTurck flow and prove some estimates about the distortion of the $C^0$ local mass along the flow. We also discuss the weak notion of nonnegative scalar curvature that we use in this paper, and we record some functional analytic results that we use in subsequent sections. In Section \ref{sec:approximationproof} we prove Theorem \ref{thm:masssequence}. In Appendix \ref{appendix:RDTFweightedXnorm} we prove a local stability result for Ricci-DeTurck flow, which may be of independent interest.

\noindent \textbf{Acknowledgements.} I had helpful discussions about this topic with many people, whom I would like to thank. They are: Lan-Hsuan Huang, Florian Johne, Sven Hirsch, Bob Haslhofer, Marcus Khuri, Claude LeBrun, Richard Bamler, Otis Chodosh, Natasa Sesum, Jeff Jauregui, and Mattia Fogagnolo. I would also like to thank Dan Lee for writing the excellent book \emph{Geometric Relativity} \cite{Lee19}. This work was supported in part by a grant of access to OpenAI models through the ChatGPT for Academic Researchers program.

\noindent \textbf{AI Use.} AI tools were used for the completion of this paper in the following ways: GPT-6 Astra was used to perform a literature search. The strategy of the proof of Lemma \ref{lemma:RDTFC0AF} is based on a strategy suggested by GPT-6 Astra, though the proof was carried out and written up by the author. The idea to use an exponential weight similar to the one that appears in Definition \ref{def:weightednorms} was partially inspired by an exchange with ChatGPT 5.6, though the actual form of the weight used in this paper and the proofs of the corresponding weighted estimates are due to the author. This paper was not written by AI.

\section{Preliminaries}\label{sec:preliminaries}
\subsection{Ricci and Ricci-DeTurck flow preliminaries}
If $M$ is a smooth manifold and $(\tilde g_t)_{t\in (0, T)}$ is a smooth family of Riemannian metrics on $M$, recall that $\tilde g_t$ evolves by Ricci flow if
\begin{equation}\label{eq:RF}
\partial_t \tilde g_t = -2\Ric(\tilde g_t).
\end{equation}
We use the notation $\tilde g_t$ to distinguish this flow from the Ricci-DeTurck flow, which we use more often in this paper, and which we will denote by $g_t$. The Ricci-DeTurck flow, introduced by DeTurck in \cite{DeTurck83}, is a strongly parabolic flow that is related to the Ricci flow by pullback via a family of diffeomorphisms. More specifically, we define the following operator, which maps symmetric $2$-forms on $M$ to vector fields:
\begin{equation}\label{eq:Xoperator}
X_{\bar g}(g):= \sum_{i=1}^n(\nabla^{\bar g}_{e_i}e_i - \nabla^{g}_{e_i}e_i),
\end{equation}
where $\{e_i\}_{i=1}^n$ is any local orthonormal frame with respect to $g$. Then the Ricci-DeTurck equation is
\begin{equation}\label{eq:RDTF}
\partial_t g(t) = -2\Ric(g(t)) - \L_{X_{\bar g(t)}(g(t))}g(t),
\end{equation}
where $\bar g(t)$ is a background Ricci flow. 

In this paper we work with Ricci-DeTurck flows on $\R^n$ with respect to a Euclidean background, so we take $\bar g(t)\equiv \delta$ and (\ref{eq:RDTF}) becomes
\begin{equation}\label{eq:RDTFeucl}
\partial_t g(t) = -2\Ric(g(t)) - \L_{X_{\delta}(g(t))}g(t).
\end{equation}
As mentioned, if $g(t)$ solves (\ref{eq:RDTF}) then it is related to a Ricci flow via pullback by diffeomorphisms. More precisely, if $g(t)$ solves (\ref{eq:RDTF}) and $(\chi_t)_{t\in (0, T)}: M\to M$ is the family of diffemorphisms satisfying
\begin{equation}\label{eq:diffeoseq}
\begin{cases}
X_{\bar g(t)}(g(t))f &= \frac{\partial}{\partial t}(f\circ\chi_t) \text{ for all } f\in C^\infty(M)\\
\chi_{\bar t} &= \id,
\end{cases}
\end{equation}
 then $\tilde g(t):= \chi_t^*g(t)$ solves (\ref{eq:RF}) with the condition $\tilde g(\bar t) = g(\bar t)$.

It is known (see \cite[Appendix A]{BamlerKleiner22}) that if $g_t$ solves (\ref{eq:RDTF}) with respect to the background Ricci flow $\bar g_t$, and if $h_t  = g_t - \bar g_t$, then the evolution equation for $h_t$ is given by
\begin{equation}\label{eq:hevolution}
\partial_t h_t = -L h_t + Q[h_t],
\end{equation}
where 
\begin{equation}\label{eq:Lis}
\begin{split}
L h_t &= -\Delta^{\bar g_t}h_t - 2\Rm^{\bar g_t}[h_t]
\\&:= -\Delta^{\bar g_t}h_t - 2{\bar g}^{pq}R_{pij}^{m}h_{q m}dx^i\otimes dx^j + \bar g^{pq}(h_{pj}R_{qi} + h_{ip}R_{qj})
\end{split}
\end{equation}
(note that our notation convention for $\Rm^{\bar g_t}[h_t]$ differs slightly from that of \cite{BamlerKleiner22} as we do not use the Uhlenbeck trick in this paper) and $Q$ denotes the quadratic term 
\begin{equation}
\begin{split}
\left(Q_{\bar g_t}[h_t]\right)_{ij}&:= \left((\bar g+h)^{pq} - \bar g^{pq}\right)\left(\nabla^2_{pq}h_{ij} +R_{pij}^mh_{mq} + R_{pji}^mh_{mq}\right)
\\& \qquad + \left(\bar g^{pq} - (\bar g+h)^{pq}\right)\left(R_{ipq}^mh_{mj} + R_{jpq}^mh_{im}\right)
\\& -\frac{1}{2}(\bar g+h)^{pq}(\bar g+h)^{m\ell}\big(-\nabla_i h_{pm}\nabla_jh_{q\ell} - 2\nabla_mh_{ip}\nabla_qh_{j\ell}
\\& \qquad + 2\nabla_mh_{ip}\nabla_{\ell}h_{jq} + 2\nabla_ph_{i\ell}\nabla_jh_{qm} + 2\nabla_ih_{pm}\nabla_qh_{j\ell}\big)
\\&= \nabla_p (\left((\bar g+h)^{pq} - \bar g^{pq}\right)\nabla_qh_{ij}) 
\\& \qquad - \left(\nabla_p\left((\bar g+h)^{pq} - \bar g^{pq}\right)\right)\nabla_qh_{ij} +  \left((\bar g+h)^{pq} - \bar g^{pq}\right)\left(R_{pij}^mh_{mq} + R_{pji}^mh_{mq}\right)
\\& \qquad + \left(\bar g^{pq} - (\bar g+h)^{pq}\right)\left(R_{ipq}^mh_{mj} + R_{jpq}^mh_{im}\right)
\\& -\frac{1}{2}(\bar g+h)^{pq}(\bar g+h)^{m\ell}\big(-\nabla_i h_{pm}\nabla_jh_{q\ell} - 2\nabla_mh_{ip}\nabla_qh_{j\ell}
\\& \qquad + 2\nabla_mh_{ip}\nabla_{\ell}h_{jq} + 2\nabla_ph_{i\ell}\nabla_jh_{qm} + 2\nabla_ih_{pm}\nabla_qh_{j\ell}\big),
\end{split}
\end{equation}
where here $\nabla$ denotes the covariant derivative with respect to $\bar g_t$. The second equality follows from the Leibniz rule. 

We often write
\begin{equation}
Q[h_t] = Q^0_t + \nabla^* Q^1_t,
\end{equation}
where
\begin{equation}\label{eq:Q0is}
\begin{split}
Q^0_t&:=-\frac{1}{2}(\bar g+h)^{pq}(\bar g+h)^{m\ell}\big(-\nabla_i h_{pm}\nabla_jh_{q\ell} - 2\nabla_mh_{ip}\nabla_qh_{jm}
\\& \qquad + 2\nabla_mh_{ip}\nabla_{\ell}h_{jq} + 2\nabla_ph_{i\ell}\nabla_jh_{qm} + 2\nabla_ih_{pm}\nabla_qh_{j\ell}\big)
\\& \qquad - \left(\nabla_p((\bar g+h)^{pq} - \bar g^{pq})\right)\nabla_qh_{ij} +  \left((\bar g+h)^{pq} - \bar g^{pq}\right)\left(R_{pij}^mh_{mq} + R_{pji}^mh_{mq}\right)
\\& = (\bar g + h)^{-1}\star (\bar g + h)^{-1} \star \nabla h \star \nabla h + ((\bar g + h)^{-1} - \bar g^{-1})\star \Rm^{\bar g_t}\star h
\end{split}
\end{equation}
and
\begin{equation}\label{eq:Q1is}
\nabla^*Q^1_t:= \nabla_p (\left((\bar g+h)^{pq} - \bar g^{pq}\right)\nabla_qh_{ij}) = \nabla(((\bar g + h)^{-1} - \bar g^{-1})\star \nabla h),
\end{equation}
where we use the notation $A\star B$ for two tensor fields $A$ and $B$ to mean a linear combination of products of the coefficients of $A$ and $B$, and $(\bar g + h)^{-1}$ and $\bar g^{-1}$ denote tensor fields with coefficients $(\bar g + h)^{ij}$ and $\bar g^{ij}$ respectively. On $\R^n$ we use $A*B$ to denote any term such that $|A*B|_\delta \leq c(n)|A|_\delta |B|_\delta$, where $\delta$ denotes the Euclidean metric.

Henceforth we take all covariant derivatives and measure all balls and all norms with respect to the Euclidean metric $\delta$ on $\R^n$, unless otherwise stated. When $\bar g(t)\equiv \delta$, (\ref{eq:hevolution}) becomes (see \cite[(4.4)]{KochLamm12}): 
\begin{equation}\label{eq:heveuclbackground}
\begin{split}
\partial_t h_{ij} &= \Delta h_{ij} + \frac{1}{2}(\delta + h)^{pq}(\delta + h)^{m\ell}\big( \nabla_i h_{pm}\nabla_j h_{q\ell} + 2\nabla_m h_{ip}\nabla_q h_{jm} -2\nabla_m h_{ip}\nabla_{\ell}h_{jp} 
\\& - 2\nabla_p h_{i\ell}\nabla_j h_{qm} -2\nabla_i h_{pm}\nabla_q h_{j\ell} \big) - \nabla_p((\delta + h)^{pq})\nabla_q h_{ij}
\\& + \nabla_p \big( ((\delta + h)^{pq} - \delta^{pq})\nabla_q h_{ij} \big)
\\&=: \Delta h_{ij} + Q^0[h] + \nabla^*Q^1[h]
\\& \text{ where}
\\& Q^0[h] = \frac{1}{2}(\delta + h)^{pq}(\delta + h)^{m\ell}\big( \nabla_i h_{pm}\nabla_j h_{q\ell} + 2\nabla_m h_{ip}\nabla_q h_{jm} -2\nabla_m h_{ip}\nabla_{\ell}h_{jp} 
\\& - 2\nabla_p h_{i\ell}\nabla_j h_{qm} -2\nabla_i h_{pm}\nabla_q h_{j\ell} \big) - \nabla_p((\delta + h)^{pq})\nabla_q h_{ij} = \nabla h * \nabla h,
\\& \nabla^*Q^1[h] = \nabla_p \big( ((\delta + h)^{pq} - \delta^{pq})\nabla_q h_{ij} \big) = \nabla (h * \nabla h),
\end{split}
\end{equation}
where $\Delta$ denotes the usual Euclidean Laplacian.

In particular, if $h$ solves (\ref{eq:hevolution}) on a Euclidean background with initial data $h_0$, then it also solves the integral equation
\begin{equation}\label{eq:integraleeq}
\begin{split}
h(x,t) &= \int_{\R^n}\Phi(x,t;y,0)h_0(y)dy 
\\& + \int_0^t\int_{\R^n} \Phi(x,t;y,s)Q^0[h](y,s) + \nabla^* \Phi(x,t;y,s)Q^1[h](y,s)dyds,
\end{split}
 \end{equation}
 where $\Phi$ denotes the Euclidean heat kernel for $(0,2)$-tensor fields, and
 \begin{equation*}
 \begin{split}
 |Q^0[h]| & = |\nabla h * \nabla h| \leq c(n)|\nabla h|^2\\
 |Q^1[h]| & = |h * \nabla h| \leq c(n)|h| |\nabla h|.
 \end{split}
 \end{equation*}

\begin{definition}\label{def:variousnorms}
For a time-dependent $(0,2)$-tensor field $h$ on $\R^n$ we define the following norm, as in \cite{KochLamm12}:
\begin{equation*}
\begin{split}
|| h ||_X &:= \sup_{0 < t < \infty}|| h(t) ||_{L^\infty(\R^n)} 
\\& + \sup_{x\in \R^n}\sup_{0 < r}\left( r^{-n/2}||\nabla h ||_{L^2(B(x,r)\times (0, r^2))} + r^{\tfrac{2}{n+4}}|| \nabla h||_{L^{n+4}(B(x, r)\times(\tfrac{r^2}{2}, r^2))} \right).
\end{split}
\end{equation*}
\end{definition}

We now record the following result concerning Ricci-DeTurck flows starting from small $C^0$ perturbations of Euclidean space (cf. \cite{Simon02}, \cite[Theorem $4.3$]{KochLamm12}, \cite[Lemma $3.3$ and Corollary $3.4$]{PBG19}, \cite{CaiWang26}):
\begin{lemma}\label{lemma:KL}
There exists $ \varepsilon = \varepsilon(n) < 1$ and $c = c(n)$ such that the following is true:

If $g_0$ is any continuous Riemannian metric on $\R^n$ such that $|| g_0 - \delta||_{C^0(\R^n)} < \varepsilon$ then there exists a smooth solution $(g_t)_{t>0}$ to (\ref{eq:RDTFeucl}) such that 
\begin{equation}\label{eq:RDTFinitialcondition}
g_t \xrightarrow[t\searrow 0]{C^0_{\loc}} g_0,
\end{equation}
\begin{equation}\label{eq:RDTFXest}
|| g_t - \delta||_{X} \leq c|| g_0 - \delta||_{C^0(\R^n)},
\end{equation}
where $||\cdot||_X$ as in Definition \ref{def:variousnorms}, and, for all $k\in \N$ there exists $c_k(n)>0$ such that for all $t>0$,
\begin{equation}\label{eq:RDTFderivests}
|| \nabla^k(g_t - \delta)||_{C^0(\R^n)} \leq c_k(n)\frac{|| g_0 - \delta||_{C^0(\R^n)}}{t^{k/2}}.
\end{equation}
\end{lemma}
\begin{proof}
The existence of a smooth solution $(g_t)_{t>0}$ and the estimates (\ref{eq:RDTFXest}) and (\ref{eq:RDTFderivests}) are due to \cite[Theorem $4.3$]{KochLamm12}. That the solution converges to the initial data as $t\searrow 0$ follows in the same way as in the proof of \cite[Corollary $3.7$]{PBG19} (also see \cite[Lemma 2.7]{PBG26}).
\end{proof}
\begin{remark}\label{rmk:parabolicrescaling}
If $g_t$ is a Ricci-DeTurck flow starting from $g_0$ in the sense of Lemma \ref{lemma:KL}, then the parabolically rescaled flow $\hat g_t(x):= g_{r^2t}(rx)$ is a Ricci-DeTurck flow in the sense of Lemma \ref{lemma:KL} starting from $\hat g_0(x):= g_0(rx)$, that is, $\hat g_t$ solves (\ref{eq:RDTFeucl}) and satisfies (\ref{eq:RDTFinitialcondition}), (\ref{eq:RDTFXest}), and (\ref{eq:RDTFderivests}) with $g_0$ replaced by $\hat g_0$.
\end{remark}

Henceforth we use the following definition:
\begin{definition}\label{def:C0AFRn}
Let $g$ be a $C^0$ Riemannian metric on $\R^n$. We say that $g$ is asymptotically flat with decay rate $\tau$ if there exists some $c_0 >0$ such that for all $x\in \R^n\setminus \overline{B(0,1)}$ we have
\begin{equation}\label{eq:C0AFisolated}
|g_{ij} - \delta_{ij}|(x) \leq c_0 |x|^{-\tau}.
\end{equation}
\end{definition}
\begin{remark}
For a more complete $C^0$ analogy with the classical asymptotically flat setting, one could also allow metrics that satisfy (\ref{eq:C0AFisolated}) after pushing forward by a diffeomorphism defined $\R^n\setminus K\to \R^n\setminus \overline{B(0,1)}$ for some compact set $K$. For simplicity we do not consider that case in this paper.
\end{remark}

\begin{lemma}\label{lemma:RDTFC0AF}
 Let $\varepsilon$ be as in Lemma \ref{lemma:KL}. There exists $\varepsilon' \leq \varepsilon$ depending only on $n$ and $\tau_0 \in (\tfrac{n-2}{2}, n-2)$ such that the following is true: suppose $g$ is a $C^0$-asymptotically flat Riemannian metric on $\R^n$ with decay rate $\tau \in (\tfrac{n-2}{2}, \tau_0]$ such that $|| g- \delta||_{C^0(\R^n)}< \varepsilon'$. Let $(g_t)_{t>0}$ denote the Ricci-DeTurck flow starting from $g$ whose existence is given by Lemma \ref{lemma:KL}. Then there exists $C= C(n)$ and $k = k(n,\tau_0)$ such that for all $t>0$,
\begin{equation*}
|| \rho_t^{\tau} (g_t - \delta)||_{L^\infty(\R^n)} \leq C||\rho_0^{\tau} (g-\delta)||_{L^\infty(\R^n)},
\end{equation*}
where $C_n$ depends only on $n$ and $\rho_t(x) = \sqrt{1 + |x|^2 + kt}$. In particular, $g_t$ is $C^0$-asymptotically flat with decay rate $\tau$ for all $t>0$.
\end{lemma}
\begin{remark}
Note that if $g$ is $C^0$-asymptotically flat with decay rate greater than or equal to $n-2$, then certainly $g$ is also $C^0$-asymptotically flat with decay rate $\tau$ for any $\tau \in (\tfrac{n-2}{2}, n-2)$.
\end{remark}

\begin{proof}
Let $\varepsilon' = \varepsilon$. We decrease $\varepsilon'$ as needed in the course of the proof. Note that the coefficient matrix $g_t^{ij}$ satisfies $||g_t^{ij} - \delta^{ij}||_{L^\infty(\R^n\times (0,\infty))} \leq C_n\varepsilon'$ by Lemma \ref{lemma:KL} and the definition of $||\cdot||_X$. As in the statement of the theorem, fix $\tau_0\in (\tfrac{n-2}{2}, n-2)$ and $\tau \in (\tfrac{n-2}{2}, \tau_0]$ and let $\rho_t(x) = \sqrt{1 + |x|^2 + kt}$, where $k$ will be chosen later. We show that $\rho_t^{-\tau}$ is a supersolution of $\partial_t u - g_t^{ij}\partial_i\partial_j u = 0$. We have
\begin{equation*}
\partial_i \rho_t(x) = \frac{x^i}{\rho_t(x)}
\end{equation*}
so
\begin{equation*}
\partial_i\partial_j \rho_t(x) = \frac{\delta_{ij}}{\rho_t(x)} - \frac{x^ix^j}{\rho_t(x)^3}
\end{equation*}
and
\begin{align*}
\Delta^{\delta} \rho_t(x)^{-\tau} &= \sum_{i=1}^{n}\partial_i \left[ (-\tau )\rho_t(x)^{-\tau -1} \partial_i\rho_t(x)\right]
\\&= \sum_{i=1}^{n}(-\tau )(-\tau  - 1)\rho_t(x)^{-\tau  - 2}(\partial_i \rho_t(x))^2 + (-\tau )\rho_t(x)^{-\tau  - 1}\partial_i^2\rho_t(x)
\\&= (-\tau )(-\tau  - 1)\rho_t(x)^{-\tau  - 2}\frac{|x|^2}{\rho_t(x)^2} + (-\tau )\rho_t(x)^{-\tau  - 1}\left[ \frac{n}{\rho_t(x)} - \frac{|x|^2}{\rho_t(x)^3}\right]
\\& = (-\tau )\rho_t(x)^{-\tau  - 4}\left[ (-\tau  - 1)|x|^2 + n\rho_t(x)^2 - |x|^2  \right]
\end{align*}

Also
\begin{align*}
\partial_t \rho_t(x)^{-\tau } &= (-\tau )\rho_t(x)^{-\tau  - 1}\frac{k}{2\rho_t(x)}
\end{align*}
so
\begin{align*}
(\partial_t - \Delta^\delta) \rho_t(x)^{-\tau } &= (-\tau )\rho_t(x)^{-\tau  - 1}\frac{k}{2\rho_t(x)} - (-\tau )\rho_t(x)^{-\tau  - 4}\left[ (-\tau  - 1)|x|^2 + n\rho_t(x)^2 - |x|^2  \right]
\\&= (-\tau )\rho_t(x)^{-\tau  - 4}\left[ \frac{k}{2}\rho_t(x)^2 + (\tau  + 1)|x|^2 - n\rho_t(x)^2 + |x|^2 \right]
\\&= (-\tau )\rho_t(x)^{-\tau  - 4}\left[(k/2 - n)\rho_t(x)^2 + (\tau  + 2)|x|^2  \right].
\end{align*}

Therefore,
\begin{align*}
(\partial_t - g^{ij}\partial_i\partial_j) \rho_t(x)^{-\tau } &= (\partial_t - \Delta^\delta)\rho_t(x)^{-\tau } + (\Delta^\delta - g^{ij}\partial_i\partial_j)\rho_t(x)^{-\tau }
\\& \geq (-\tau )\rho_t(x)^{-\tau  - 4}\left[(k/2 - n)\rho_t(x)^2 + (\tau  + 2)|x|^2  \right] + c_n\varepsilon'(-\tau )\rho_t(x)^{-\tau  -2}
\\& = (-\tau )\rho_t(x)^{-\tau  - 4} \left[(k/2 - n + c_n \varepsilon')\rho_t(x)^2 + (\tau  + 2)|x|^2  \right] \geq 0
\end{align*}
where in the last step we are using that $\tau  + 2 < n$, so there exists some small positive $k = k(n, \tau)$ such that $k/2 - n + c_n\varepsilon' + \tau + 2 <0$, and $\rho_t(x) \geq |x|$. 

Now note that by \cite[(4.3)]{KochLamm12}, if $|| g_t - \delta||_{L^\infty(\R^n\times(0,\infty))} < 1/2$, we then have, for $h_t = g_t - \delta$, that 
\begin{equation}
\partial_t h_t - g_t^{ij}\partial_i\partial_j h_t = \bar Q[h_t],
\end{equation}
where 
\begin{equation}\label{eq:forcingbound}
|\bar Q[h_t]| \leq C_n|\nabla h_t|^2.
\end{equation}
 Let $\kappa = 2C_n$ (for the particular value of $C_n$ in (\ref{eq:forcingbound})). Reduce $\varepsilon'$ as needed so that by Lemma \ref{lemma:KL}, $||h_t||_{L^\infty(\R^n\times(0,\infty))} \leq c_n \varepsilon' < \min\{\tfrac{1}{2}, \tfrac{1}{4C_n}\}$ and for all $\xi \in \R^n, x\in \R^n, t>0$, $g_t^{ij}\xi_i\xi_j \geq \tfrac{1}{2}|\xi|^2$. For $\eta>0$, let $F_\eta(x,t) := e^{\kappa \sqrt{|h_t|^2 + \eta^2}} - e^{\kappa \eta}$. We will show that $F_\eta$ is a subsolution of $\partial_t u - g_t^{ij}\partial_i\partial_j u =0$ for all sufficiently small $\eta$. First, let $a_\eta = \sqrt{|h|^2 + \eta^2}$. Then 
\begin{align*}
(\partial_t - g_t^{ij}\partial_i\partial_j) a_\eta &= \frac{\langle h, \partial_t h \rangle_\delta}{a_\eta} - g_t^{ij}\left(\frac{\langle \partial_i h, \partial_i j\rangle_\delta}{a_\eta} + \frac{\langle h, \partial_i \partial_j h\rangle_\delta}{a_\eta} - \frac{\langle h, \partial_j h\rangle_\delta\langle h, \partial_i h\rangle_\delta}{a_\eta^3}\right)
\\&= \frac{\langle h, (\partial_t - g_t^{ij}\partial_i \partial_j)h \rangle_\delta}{a_\eta} -\frac{g_t^{ij}\langle \partial_i h, \partial_j h\rangle_\delta}{a_\eta} + \frac{g_t^{ij}(\partial_j a_\eta)(\partial_i a_\eta)}{a_\eta}.
\end{align*}

Therefore, we have
\begin{align*}
(\partial_t - g_t^{ij}\partial_i\partial_j)F_\eta &= \kappa e^{\kappa a_\eta}\left(\partial_t a_\eta - g_t^{ij}\partial_i\partial_j a_\eta \right) - \kappa^2e^{\kappa a_\eta}g_t^{ij}(\partial_i a_\eta)(\partial_j a_\eta)
\\&= \kappa e^{\kappa a_\eta}\left[\frac{\langle h, \bar Q[h] \rangle_\delta}{a_\eta} -\frac{g_t^{ij}\langle \partial_i h, \partial_j h\rangle_\delta}{a_\eta} + \frac{g_t^{ij}(\partial_j a_\eta)(\partial_i a_\eta)}{a_\eta}\right] - \kappa^2e^{\kappa a_\eta}g_t^{ij}(\partial_i a_\eta)(\partial_j a_\eta)
\\&\leq \kappa e^{\kappa a_\eta}\left[|\bar Q[h]| -\frac{g_t^{ij}\langle \partial_i h, \partial_j h\rangle_\delta}{a_\eta} + \frac{g_t^{ij}(\partial_j a_\eta)(\partial_i a_\eta)}{a_\eta} \right] - \kappa^2e^{\kappa a_\eta}g_t^{ij}(\partial_i a_\eta)(\partial_j a_\eta)
\\& \leq \kappa e^{\kappa a_\eta}\left[\kappa g_t^{ij}\langle \partial_i h, \partial_j h\rangle_\delta -\frac{g_t^{ij}\langle \partial_i h, \partial_j h\rangle_\delta}{a_\eta} + \frac{g_t^{ij}(\partial_j a_\eta)(\partial_i a_\eta)}{a_\eta} - \kappa g_t^{ij}(\partial_i a_\eta)(\partial_j a_\eta) \right]
\\&= \kappa e^{\kappa a_\eta}\left(\kappa - \frac{1}{a_\eta}\right)\left(g_t^{ij}\langle \partial_i h, \partial_j h\rangle_\delta - g_t^{ij}(\partial_j a_\eta)(\partial_i a_\eta) \right)
\end{align*}
where in third step we are using Cauchy--Schwarz and the fact that $|h| \leq a_\eta$, and in the fourth step we are using (\ref{eq:forcingbound}), the definition of $\kappa$, and the fact that $g_t$ is uniformly bilipschitz to $\delta$.

 Therefore, $F_\eta$ is a subsolution of $\partial_t u - g_t^{ij}\partial_i\partial_j u =0$ for small $\eta$ provided that we can show $\kappa a_\eta < 1$ for sufficiently small $\eta$ and that $g_t^{ij}\langle \partial_i h, \partial_j h\rangle_\delta - g_t^{ij}(\partial_j a_\eta)(\partial_i a_\eta) \geq 0$. Towards the first inequality, note that for $\eta < \varepsilon'$ we have
\begin{equation*}
\kappa a_\eta = 2C_n\sqrt{|h|^2 + \eta^2} \leq 2C_n \sqrt{(c_n \varepsilon')^2 + \eta^2} < 1
\end{equation*}
after reducing $\varepsilon'$ as needed depending on $n$. 
Towards the second, consider the matrix $B_{ij} = \langle\partial_i h, \partial_j h\rangle_\delta - (\partial_j a_\eta)(\partial_i a_\eta)$. Then for any $\xi \in \R^n$, Cauchy--Schwarz implies
\begin{align*}
B_{ij}\xi^i\xi^j &= \sum_{i,j=1}^{n} \xi^i\xi^j\langle\partial_i h, \partial_j h\rangle_\delta - \sum_{i,j=1}^n \xi^i\xi^j (\partial_j a_\eta)(\partial_i a_\eta)
\\&= \sum_{i,j=1}^{n} \xi^i\xi^j\langle\partial_i h, \partial_j h\rangle_\delta - \sum_{i,j=1}^n \xi^i\xi^j \frac{\langle\partial_i h, h\rangle_\delta}{a_\eta}\frac{\langle\partial_j h, h\rangle_\delta}{a_\eta}
\\& =  \left| \sum_{i=1}^{n}\xi^i\partial_i h\right|_\delta^2 - \frac{1}{a_\eta^2}\bigg \langle \sum_{i=1}^n\xi^i\partial_i h, h \bigg\rangle_\delta^2
\\& \geq \left| \sum_{i=1}^{n}\xi^i\partial_i h\right|_\delta^2 - \frac{|h|^2}{a_\eta^2}\bigg| \sum_{i=1}^n \xi^i \partial_i h \bigg|_\delta^2
\\& \geq 0,
\end{align*}
where in the last step we are using that $|h| \leq a_\eta$. In particular, the matrix $B$ is positive semi-definite, so 
\begin{equation*}
g_t^{ij}\langle \partial_i h, \partial_j h\rangle_\delta - g_t^{ij}(\partial_j a_\eta)(\partial_i a_\eta) = g_t^{ij}B_{ij} = \tr([g_t]^{-1}B) \geq 0.
\end{equation*}
This shows that $F_\eta$ is a subsolution of $\partial_t u - g_t^{ij}\partial_i\partial_j u =0$ for small $\eta$. Since $F_\eta$ is a subsolution and $C' \rho^{-\tau}$ is a supersolution of $\partial_t u - g_t^{ij}\partial_i\partial_j u =0$ for any $C'>0$, we will see below that the comparison principle implies, for all $t>0$, $x\in \R^n$ and $\eta$ sufficiently small,
\begin{equation*}
F_\eta(x,t) \leq C'\rho_t^{-\tau}(x),
\end{equation*}
provided that $C'$ is chosen so that 
\begin{equation*}
F_\eta(x, 0) \leq C' \rho_0^{-\tau}(x)
\end{equation*}
for all $x\in \R^n$. To achieve this latter condition, let
\begin{equation*}
C' = \kappa e^{2\kappa \varepsilon'}|| \rho_0^{\tau}h_0||_{L^\infty(\R^n)}.
\end{equation*}
Reduce $\varepsilon'$ as needed so that $\sqrt{2\varepsilon '} < 1$. Then, for $\eta \leq \varepsilon'$ we have by the Mean Value Theorem
\begin{align*}
F_\eta(x, 0) & = \exp\left(\kappa \sqrt{|h_0|^2 + \eta^2}\right) - e^{k\eta}
\\& \leq \kappa \exp\left(\kappa \sqrt{2(\varepsilon')^2}\right)(\sqrt{|h_0(x)|^2 + \eta^2} - \eta)
\\& \leq \kappa \exp\left(\kappa 2\varepsilon'\right)|h_0(x)| \leq C' \rho_0^{-\tau }(x).
\end{align*}
Then one may apply the weak maximum principle \cite[Chapter 2, Section 4, Theorem 9]{Friedman64} to $C'\rho_t^{-\tau}(x) - F_\eta(x)$ to find that for $\eta$ sufficiently small and all $x\in \R^n$, $t>0$, we have
\begin{equation*}
F_\eta(x,t) \leq C'\rho_t^{-\tau}(x).
\end{equation*}

 Fixing $x\in \R^n$ and $t>0$ and letting $\eta\to 0$, we find that
\begin{align*}
\sup_{x\in \R^n, t>0} \rho_t^{\tau}(x) (e^{\kappa |h(x, t)|} - 1) & \leq c_n || \rho_0^{\tau}h_0||_{L^\infty(\R^n)}.
\end{align*}
Moreover, for any $x\in \R^n$ and $t>0$, we have $\kappa |h(x,t)| \leq e^{\kappa |h(x, t)|} - 1$ so
\begin{equation*}
\sup_{x\in \R^n, t>0} \rho^{\tau}(x) |h(x,t)| \leq c_n || \rho_0^{\tau}h_0||_{L^\infty(\R^n)}.
\end{equation*}
This proves the result. 
\end{proof}

\begin{corollary}\label{cor:smoothdecay}
Let $\varepsilon'$ be as in Lemma \ref{lemma:RDTFC0AF}. For $n\geq 3$ and $\tau_0 \in (\tfrac{n-2}{2}, n-2)$ there exists $\varepsilon'' = \varepsilon''(n, \tau_0) \leq \varepsilon'$ for which the following is true:

Let $g$ be a $C^0$-asymptotically flat Riemannian metric on $\R^n$ with decay rate $\tau \in (\tfrac{n-2}{2}, \tau_0]$, such that $|| g - \delta||_{L^\infty(\R^n)} \leq \varepsilon''$. Let $g_t$ be the Ricci-DeTurck flow starting from $g$ in the sense of Lemma \ref{lemma:KL}. Then for $m= 1, 2, 3, 4$,  there exists $c = c(n, m, || \rho_0^\tau (g-\delta)||_{L^\infty(\R^n)}, \tau)$,  where $\rho_t$ is as in Lemma \ref{lemma:RDTFC0AF}, such that for all $x \in \R^n$ and all $t>0$ we have
\begin{equation*}
|\nabla^m g_t|(x) \leq \frac{c_m \rho_t(x)^{-\tau}}{t^{m/2}} \leq  \frac{c_m |x|^{-\tau}}{t^{m/2}}.
\end{equation*}
\end{corollary}
\begin{proof}
Let $\varepsilon'' = \varepsilon_2$ from \cite[Proposition 2.5]{Bamler14}. Fix some $x\in \R^n$ and $t>0$. Note that if $(y,s)\in B(x, \sqrt{t})\times (\tfrac{7}{8}t, \tfrac{9}{8}t)$ we have
\begin{align*}
\rho_s(y)^2 & = 1 + |y|^2 + ks \geq 1 + (|x| - \sqrt{t})^2 + \tfrac{7}{8}kt
\\& \geq 1 + |x|^2 + \frac{7k + 8}{8} \geq c(n,\tau)\rho_t(x)^2.
\end{align*}

By \cite[Proposition 2.5]{Bamler14} and Lemma \ref{lemma:RDTFC0AF} we have
\begin{align*}
\sqrt{t}^m|| \nabla^m g_t||_{B(x, \sqrt{t}/2)\times (\tfrac{7}{8}t, \tfrac{9}{8}t)} & \leq C(n,m)||g_t - \delta ||_{C^0(B(x, \sqrt{t})\times (\tfrac{1}{8}t, \tfrac{9}{8}t))} 
\\& \leq C\sup_{(y,s)\in B(x, \sqrt{t})\times \tfrac{1}{8}t, \tfrac{9}{8}t)} ||\rho_0^\tau (g - \delta)||_{L^\infty(\R^n)}\rho_s(y)^{\tau}
\\& \leq C\rho_t(x)^{-\tau},
\end{align*}
with $C$ adjusted. This proves the result.
\end{proof}

Corollary \ref{cor:smoothdecay} alone is actually not sufficient for our purposes. In order to bound the distortion of the mass along the Ricci-DeTurck flow, we also require the following $W^{1,2}$-estimate which follows from a weighted version of the $X$-norm estimate in Lemma \ref{lemma:KL}, which we will prove in Appendix \ref{appendix:RDTFweightedXnorm}.
\begin{lemma}\label{lemma:RDTFL2distortion}
Let $g$ be a $C^0$-asymptotically flat Riemannian metric on $\R^n$ with decay rate $\tau$ such that $|| g - \delta|| \leq \varepsilon'''$, where $\varepsilon'''$ is the constant $\varepsilon$ from Theorem \ref{thm:weighteddifference}. Let $(g_t)_{t>0}$ be the Ricci-DeTurck flow starting from $g$ in the sense of Lemma \ref{lemma:KL}. Let $0 < \eta < 2$. Then there exists $C = C(n, || |\cdot|^\tau (g-\delta)||_{L^\infty(\R^n)}, \tau, \eta)$ and $\bar r = \bar r(n, \tau, \eta) > 1$ such that for all $r \geq \bar r$ the following is true:

For $x\in A(0, .9r, 1.1r)$ we have
\begin{equation*}
r^{-n/2}||\nabla (g_t - \delta)||_{L^2(B(x,.1r)\times (0, r^{2-\eta}))} \leq C r^{-\tau}.
\end{equation*}
\end{lemma}
We prove Lemma \ref{lemma:RDTFL2distortion} in Appendix \ref{appendix:RDTFweightedXnorm}.

\subsection{Nonnegative scalar curvature for $C^0$ metrics}\label{subsec:equivdefs}
We now briefly explain the weak notion of nonnegative scalar curvature that we are using. For a more exhaustive discussion of the development of this notion, we refer the reader to \cite{PBG19}, \cite{PBG20}, and \cite[Definition 2.3]{PBG23}.
\begin{definition}\label{def:RFNNSC}
Let $M^n$ be a smooth manifold and $g$ be a continuous Riemannian metric on $M$. For $\beta \in (0,1/2)$ we say that $g$ has scalar curvature bounded below by $\kappa_0\in \R$ in the $\beta$-weak sense at $x\in M$ if there exists a coordinate chart $\Phi: U_x\to \Phi(U_x)$ for $M$, where $U_x$ is a neighborhood of $x$, and there exists a continuous metric $g_0$ on $\R^n$ and a Ricci-DeTurck flow $(g_t)_{t\in (0,T]}$ for $g_0$, with respect to the Euclidean background metric, satisfying (\ref{eq:RDTFinitialcondition}), (\ref{eq:RDTFXest}), and (\ref{eq:RDTFderivests}), such that
\begin{equation*}
g_0 \big|_{\Phi(U_x)} = \Phi_*g,
\end{equation*}
and such that
\begin{equation}\label{eq:betaweakcondition}
\inf_{C>0}\left(\liminf_{t\searrow 0}\left(\inf_{B_{\delta}(\Phi(x), Ct^{\beta})} R(g_t) \right)\right) \geq \kappa_0.
\end{equation}
We say that $g$ has scalar curvature bounded below by $\kappa_0\in \R$ in the sense of Ricci flow on an open region $U\subset M$ if there exists some $\beta\in (0, 1/2)$ such that, for all $x\in U$, $g$ has scalar curvature bounded below by $\kappa_0$ in the $\beta$-weak sense at $x$.
\end{definition}

\begin{lemma}\label{lemma:RFNNSC}
Let $\varepsilon$ be as in Lemma \ref{lemma:KL} and suppose $g$ is a continuous Riemannian metric on $\R^n$ such that $|| g_0 - \delta||_{C^0(\R^n)} < \varepsilon$. Assume that $g$ has nonnegative scalar curvature in the sense of Ricci flow. If $(g_t)_{t>0}$ is the Ricci-DeTurck flow starting from $g$ whose existence is given by Lemma \ref{lemma:KL} then $R(g_t)\geq 0$ for all $t>0$.
\end{lemma}
This is due to \cite[Remark 2.4]{PBG23}.

\subsection{$C^0$ mass}
We now define the quantity in Theorem \ref{thm:C0ADMmassexistence}:
\begin{definition}\label{def:C0mass}
Let $r>0$, $g$ be a $C^0$ metric on a smooth manifold $M$, $\Phi: U\to \R^n$ a smooth coordinate chart for $M$ such that $A(0, .9r, 1.1r) \subset \Phi(U)$, where $A(0, .9r, 1.1r)$ denotes the annulus in $\R^n$ measured with respect to the Euclidean metric, and $\varphi:\R\to \R$ some smooth function such that $\int_{.9}^{1.1}\varphi(s)ds \neq 0$. Writing $g_{ij}$ for $(\Phi_*g)_{ij}$, we define the \emph{$C^0$ local mass of $g$ with respect to $\varphi$ and $\Phi$ at $r$} by
\begin{equation}\label{eq:defC0mass}
\begin{split}
M_{C^0}&(g, \Phi,  \varphi, r)
\\& := 
\frac{(4\pi (n-1)\omega_{n-1})^{-1}}{r\int_{.9}^{1.1}\varphi(\ell)d\ell}\sum_{i,j = 1}^{n}\bigg[ \int_{A(o,.9r, 1.1r)} \left( \frac{n-2}{|x|}\varphi(\tfrac{|x|}{r}) + \frac{1}{r}\varphi'(\tfrac{|x|}{r})\right)\delta^{ij}(g_{ij} - \delta_{ij})
\\& \qquad \qquad + \left(\frac{1}{|x|}\varphi(\tfrac{|x|}{r}) - \frac{1}{r}\varphi'(\tfrac{|x|}{r})\right)(g_{ij} - \delta_{ij})\frac{x^ix^j}{|x|^2}dx
\\& \qquad \qquad+ \int_{\partial A(o, .9r, 1.1r)} (g_{ij} - \delta_{ij})\frac{x^i}{|x|}\varphi(\tfrac{|x|}{r})\nu^j - (g_{jj}- \delta_{jj})\frac{x^i}{|x|}\varphi(\tfrac{|x|}{r})\nu^i dS\bigg],
\end{split}
\end{equation}
where $\nu$ denotes the outward unit normal vector with respect to the Euclidean metric along $\partial A(0,.9r, 1.1r)$ and $dS$ is the Euclidean surface measure on $\partial A(0, .9r, 1.1r)$.  Henceforth we write $M_{C^0}(g, \varphi, r)$ for $M_{C^0}(g, \Phi, \varphi, r)$ since, following the convention of Definition \ref{def:C0AFRn}, it is not necessary to emphasize the choice of asymptotically flat coordinate chart.
\end{definition}
\begin{remark}\label{rmk:C0agreeswithC1average}
It is a computation (see \cite[Remark 2.2]{PBG23}) to show that for any smooth $\varphi: \R \to \R$ with $\int_{.9}^{1.1}\varphi(s)ds \neq 0$,
\begin{equation*}
M_{C^0}(g,  \varphi, r) = (4\pi (n-1)\omega_{n-1})^{-1}\frac{\int_{.9r}^{1.1r}\int_{\S(u)} (\partial_j g_{ij} - \partial_i g_{jj})\nu^i\varphi(\tfrac{u}{r})dSdu}{r\int_{.9}^{1.1}\varphi(\ell)d\ell},
\end{equation*}
so $M_{C^0}(g, \varphi, r)$ is the average (weighted by $\varphi$) of the local quantity $c(n)\int_{\S(r)} (\partial_j g_{ij} - \partial_i g_{jj})\nu^idS$, which converges as $r\to \infty$ to $m_{ADM}(g)$ when latter exists.
\end{remark}

The following lemma quantifies the distortion of the $C^0$-local mass along the Ricci-DeTurck flow:
\begin{lemma}\label{lemma:massdistortionestimate}
Let $\varphi: \R \to \R^{\geq 0}$ be a smooth cutoff function with $\supp(\varphi)\subset \subset (.9, 1.1)$, and let $g_0$ be a $C^0$-asymptotically flat metric on $\R^n$ with decay rate $\tau \in (\tfrac{n-2}{2}, n-2)$ such that $|| g_0 - \delta||_{C^0(\R^n)} < \min\{\varepsilon', \varepsilon'''\}$, where $\varepsilon' = \varepsilon'(n,\tau)$ is as in Lemma \ref{lemma:RDTFC0AF} and $\varepsilon'''$ is as in Lemma \ref{lemma:RDTFL2distortion}. Let $(g_t)_{t >0}$ be the Ricci-DeTurck flow starting from $g_0$ given by Lemma \ref{lemma:KL}, and let $\bar \eta >0$. There exists $\bar r = \bar r(n, \bar \eta, \Supp(\varphi), \tau, ||\rho_0^\tau(g_0 - \delta)||_{L^\infty(\R^n)})$, where $\rho_t$ is as in Lemma \ref{lemma:RDTFC0AF}, such that for all $r > \bar r$ and $\eta > \bar \eta$, there exists a smooth solution $\varphi_{r^{-\eta}}(\ell, t): \R \times [0, r^{-\eta}]\to \R$ to
\begin{equation}\label{eq:cutofffunctionevsinglevar}
\begin{cases}
\partial_t \varphi_{r^{-\eta}}(|x|, t) &= -\Delta \varphi_{r^{-\eta}}(|x|, t) + \frac{n-1}{|x|^2}\varphi_{r^{-\eta}}(|x|, t) \text{ for } (x,t)\in \R^n\times (0, r^{-\eta})\\
\varphi_{r^{-\eta}}(\ell, r^{-\eta}) &= \varphi(\ell) \text{ for all } \ell\in \R
\end{cases}
\end{equation}
such that
\begin{equation*}
\int_0^{r^{2-\eta}}\left|  \frac{d}{dt}M_{C^0}(g_t, \varphi_{r^{-\eta}}(\cdot, \tfrac{t}{r^2}), r)\right|dt \leq c(n,\varphi, || \rho_0^\tau(g_0-\delta)||_{L^\infty(\R^n)}, \tau, \bar \eta)r^{n - 2 - 2\tau}.
\end{equation*} 
In particular, 
\begin{equation*}
M_{C^0}(g_0) = \lim_{r\to \infty}M_{C^0}(g_0, \varphi_{r^{-\eta}}(0), r) =  \lim_{r\to \infty}M_{C^0}(g_{r^{2-\eta}}, \varphi, r).
\end{equation*}
\end{lemma}
\begin{remark}
Observe that, because of the specified lower bound for $\tau$, $n-2 -2\tau <0$, so the right hand side decays to $0$ as $r\to \infty$. If $u$ is any spherically symmetric solution to the backwards heat equation $\partial_t u = -\Delta u$ on Euclidean space, then the evolution equation from (\ref{eq:cutofffunctionevsinglevar}) is precisely the evolution equation for the radial derivative of $u$.
\end{remark}
\begin{proof}[Proof of Lemma \ref{lemma:massdistortionestimate}]
The lemma essentially follows from the proof of \cite[Lemma 2.6]{PBG23}, the only difference being that rather than taking $g_0$ to be extension of a $C^0$-asymptotically flat metric restricted to the annulus $A(0, .8r, 1.2r)$, which satisfies the global bound $||g_0 - \delta||_{C^0(\R^n)} \leq c_0 r^{-\tau}$, we take $g_0$ to be a general $C^0$-asymptotically flat metric and make use of Lemma \ref{lemma:RDTFC0AF} and Lemma \ref{lemma:RDTFL2distortion}. Note that \cite[(3.13), (3.14), (3.15)]{PBG23} are given here by Lemma \ref{lemma:KL}. Let $\varphi_\theta$ be as in \cite[Lemma 4.1]{PBG23}, with prescribed final data $\varphi$, and let $D = \dist(\Supp(\varphi), \partial(.9, 1.1))$, so $D = d_{a,b}$ in the notation of \cite[Lemma 4.1]{PBG23}, and hence $D = D(\varphi)$. It then follows from Lemma \ref{lemma:RDTFC0AF} and the proof of \cite[Corollary 4.3]{PBG23} that for all $\theta_0 \in [0, r^{2-\eta})$ we have
\begin{equation*}
\begin{split}
&  \bigg|\int_{\theta_0}^{r^{2-\eta}}\frac{d}{dt}\left[M_{C^0}(g_t , \varphi_{r^{-\eta}}(\tfrac{t}{r^2}), r)\left(r\int_{.9}^{1.1}\varphi_{r^{-\eta}}(\ell, \tfrac{t}{r^2})d\ell\right)\right]dt \bigg|
\\&  \leq cr^{-2\tau}r^{n-1} + c'(n)r^{n\eta/2 + n -1}\exp\left(-\frac{D^2}{4}r^{\eta}\right),
\end{split}
\end{equation*}
as we will now explain: By the proof of \cite[Corollary 4.3]{PBG23}, 
\begin{equation*}
\begin{split}
&  \bigg|\int_{\theta_0}^{r^{2-\eta}}\frac{d}{dt}\left[M_{C^0}(g_t , \varphi_{r^{-\eta}}(\tfrac{t}{r^2}), r)\left(r\int_{.9}^{1.1}\varphi_{r^{-\eta}}(\ell, \tfrac{t}{r^2})d\ell\right)\right]dt \bigg|
\\&  \leq c'(n)r^{n\eta/2 + n -1}\exp\left(-\frac{D^2}{4}r^{\eta}\right) + A + B,
\end{split}
\end{equation*}
where, for some $z_1, z_2\in A(0, .9r, 1.1r)$,  
\begin{equation}\label{eq:ABestimate}
\begin{split}
A + B &\leq c(n,\varphi)r^{n-1}(r^{-n}||\nabla h||^2_{L^2(B(z_1, .5r)\times (0,r^{2-\eta}))}) 
\\& \qquad \qquad+ c(n, \varphi)r^{n-1}(r^{-n/2}||h||_{C^0(B(z_2, .5r)\times(0, r^{2-\eta}))}||\nabla h||_{L^2(B(z_2,.5r)\times(0,r^{2-\eta}))}). 
\end{split}
\end{equation}

We use Lemma \ref{lemma:RDTFL2distortion} to estimate these two summands. Assuming $r > \bar r(n,\tau, \eta)$, where $\bar r$ is as in Lemma \ref{lemma:RDTFL2distortion}, we may cover $B(z_i, .5r)$ with $k(n)$-many balls of radius $.1r$ and apply Lemma \ref{lemma:RDTFL2distortion} to each such ball to find that
\begin{equation*}
r^{-n/2}||\nabla h||_{L^2(B(z_i, .5r)\times (0, r^{2-\eta}))} \leq c(n, \varphi, || |\cdot|^\tau(g-\delta)||_{L^\infty(\R^n)}, \eta, \tau)r^{-\tau}.
\end{equation*}
Combining this estimate with Lemma \ref{lemma:RDTFC0AF} we find that
\begin{equation*}
\begin{split}
A + B &\leq c(n, \varphi, || |\cdot|^\tau(g-\delta)||_{L^\infty(\R^n)}, \eta, \tau)r^{n-1}r^{-2\tau}
\end{split}
\end{equation*}

Note that \cite[(4.3), (4.4)]{PBG23} still hold in this setting so combining them as in the proof of \cite[Lemma 2.6]{PBG23} yields the estimate
\begin{equation*}
\begin{split}
\int_0^{r^{2-\eta}}\bigg| \frac{d}{dt} M_{C^0}(g_t, r, \varphi(\cdot, \tfrac{t}{r^2}))\bigg|dt & \leq c(n,\varphi)r^{-2\tau}r^{n-2} + c(n, \varphi)r^{n-2 +n\eta/2 - 2 + 2-\eta}\exp\left(-\frac{D^2}{4}r^{\eta}\right)
\\& \leq c(n, \varphi, \eta, \tau)r^{-2\tau + n - 2},
\end{split}
\end{equation*}
where the last step is achieved by choosing $\bar r$ and hence $r$ sufficiently large to absorb the exponential term.
\end{proof}

Having established Lemma \ref{lemma:massdistortionestimate} we now state Theorem \ref{thm:C0ADMmassexistence} more precisely:
\begin{theorem}\label{thm:fullmassexistence}
In the setting of Lemma \ref{lemma:massdistortionestimate}, let $\varphi_{r^{-\eta}}(t)$ denote the smooth time-dependent function corresponding to $\varphi$ whose existence is given by Lemma \ref{lemma:massdistortionestimate}. Then the quantity $M_{C^0}(g, \varphi_{r^{-\eta}}(0), r)$ given by Definition \ref{def:C0mass} has the following properties:
\begin{enumerate}
\item\label{item:generalC0limitexistence} If $g$ is $C^0$-asymptotically flat with decay rate $\tau$ for some $\tau > (n-2)/2$ and $g$ has nonnegative scalar curvature in the sense of Ricci flow (Definition \ref{def:RFNNSC}) outside of a compact set, then the limit $M_{C^0}(g) := \lim_{r\to \infty}M_{C^0}(g, \varphi_{r^{-\eta}}(0), r)$ exists, is either finite or $+\infty$, and is independent of choice of $\varphi$ and choice of $C^0$-asymptotically flat coordinate chart for $g$. 
\item If $g$ is $C^2$ and $m_{ADM}(g)$ exists, then $m_{ADM}(g) = \lim_{r\to\infty}M_{C^0}(g, \varphi_{r^{-\eta}}(0), r)$.
\end{enumerate}
\end{theorem}
\begin{remark}
The limit $M_{C^0}(g)$ is also independent of choice of $\eta$; see \cite[Remark 2.10]{PBG23}.
\end{remark}

\begin{lemma}\label{lemma:Bartnikmonotonicity}
Suppose $g$ is a $C^2$ Riemannian metric on $\R^n\setminus \overline{B(0, r_0)}$ for some $r_0>0$. Then for any smooth functions $\varphi_i: \R \to \R$ such that $\int_{.9}^{1.1}\varphi_i(s)ds\neq 0$ for $i= 1,2$ and any $r_1 >1/.9 r_0$ and $r_2 > 1.1 r_1$ we have
\begin{equation*}
M_{C^0}(g, \varphi_2, r_2) - M_{C^0}(g, \varphi_1, r_1) = (4\pi(n-1)\omega_{n-1})^{-1}\int_{A(0, \ell_1 r_1, \ell_2 r_2)} R(g) + \nabla g * \nabla g
\end{equation*}
for some $\ell_1, \ell_2 \in [.9, 1.1]$.
\end{lemma}
\begin{proof}
This is essentially due to Bartnik \cite{Bartnik86}. By, say, \cite[Lemma 3.7]{PBG23} and Remark \ref{rmk:C0agreeswithC1average} we have that for $i = 1,2$ there exist $\ell_i \in [.9, 1.1]$ such that 
\begin{align*}
M_{C^0}(g, \varphi_i, r_i) &= (4\pi(n-1)\omega_{n-1})^{-1}\frac{\int_{.9}^{1.1}\int_{\S(ru)}(\partial_j g_{ij} - \partial_i g_{jj})\nu^idS\varphi_i(u)du}{\int_{.9}^{1.1}\varphi_i(\ell)d\ell}
\\& = (4\pi(n-1)\omega_{n-1})^{-1}\int_{\S(\ell-i r_i)}(\partial_j g_{ij} - \partial_i g_{jj})\nu^idS
\end{align*}
so, calculating as in \cite{Bartnik86} or \cite[(3.21)]{PBG23}, we have
\begin{equation*}
M_{C^0}(g, \varphi_2, r_2) - M_{C^0}(g, \varphi_1, r_1) = (4\pi(n-1)\omega_{n-1})^{-1}\int_{A(0, \ell_1 r_1, \ell_2 r_2)} R(g) + \nabla g * \nabla g.
\end{equation*}
\end{proof}

\begin{lemma}\label{lemma:finitemass}
Let $g$ be a $C^0$-asymptotically flat metric on $\R^n$ with decay rate $\tau \in (\tfrac{n-2}{2}, n-2)$ such that $|| g - \delta||_{C^0(\R^n)} < \varepsilon''$, where $\varepsilon'' = \varepsilon''(n, \tau)$ is as in Corollary \ref{cor:smoothdecay}. Let $\eta \in (0, 2\tau - n + 2)$ and let $(g_t)_{t>0}$ be the Ricci-DeTurck flow starting from $g$ given by Lemma \ref{lemma:KL}. Suppose $g$ has nonnegative scalar curvature in the sense of Ricci flow. If $M_{C^0}(g)$ is finite then for any $\ell_1 \in (0,1)$ and $\ell_2 > \ell_1$ we have
\begin{equation*}
\lim_{r\to\infty} \int_{A(0, \ell_1 r, \ell_2 r)} R(g_{r^{2-\eta}}) = 0.
\end{equation*}
\end{lemma}
We refer the reader to \cite[Theorem 2.9 and Theorem 7.3]{PBG23} for a stronger version of Lemma \ref{lemma:finitemass}, but for our purposes Lemma \ref{lemma:finitemass} is sufficient.
\begin{proof}
Since $M_{C^0}(g)$ is finite, 
\begin{equation*}
\lim_{r\to \infty} M_{C^0}(g, \varphi_{(\ell_1r/1.2)^{-\eta}}(0), \ell_1 r/1.2) - M_{C^0}(g, \varphi_{(\ell_2r/.8)^{-\eta}}(0), \ell_2 r/.8) = 0.
\end{equation*}
Pick $\eta' \in (0, \eta)$ so that for $r$ sufficiently large (depending on $\eta$ and $\eta'$) $r^{2-\eta} \leq (\ell_1 r/1.2)^{2-\eta'}$. 
Then, for fixed large $r$, by Lemma \ref{lemma:Bartnikmonotonicity} there exist $\ell_1', \ell_2' \in [.9, 1.1]$ such that after applying Lemma \ref{lemma:massdistortionestimate} and Corollary \ref{cor:smoothdecay} we have 
\begin{align*}
& M_{C^0}(g, \varphi_{(\ell_1r/1.2)^{-\eta'}}(0), \ell_1 r/1.2 ) - M_{C^0}(g, \varphi_{(\ell_2r/.8)^{-\eta'}}(0), \ell_2 r ) 
\\&= M_{C^0}(g_{r^{2-\eta}}, \varphi_{(\ell_1r/1.2)^{-\eta'}}(\tfrac{r^{2-\eta}}{(\ell_1 r/1.2)^2}), \ell_1 r/1.2 ) - M_{C^0}(g_{r^{2-\eta}}, \varphi_{(\ell_2r/.8)^{-\eta'}}(\tfrac{r^{2-\eta}}{(\ell_2 r/.8)^2}), \ell_2 r ) + cr^{n-2 - 2\tau}
\\& = \int_{A(0, \ell_1'\ell_1 r/1.2, \ell_2'\ell_2 r/.8)} R(g_{r^{2-\eta}}) + \nabla g_{r^{2-\eta}} * \nabla g_{r^{2-\eta}}
\\& \leq \int_{A(0, \ell_1'\ell_1 r/1.2, \ell_2'\ell_2 r/.8)} R(g_{r^{2-\eta}})dx + c(\ell_2'\ell_2 r/.8)^n\frac{(\ell_1'\ell_1 r/1.2)^{-2\tau}}{r^{2-\eta}}
\\& \leq \int_{A(0, \ell_1'\ell_1 r/1.2, \ell_2'\ell_2 r/.8)} R(g_{r^{2-\eta}})dx + cr^{n-2  + \eta - 2\tau},
\end{align*}
so, letting $r\to \infty$, we have that
\begin{equation*}
\lim_{r\to \infty } \int_{A(0, \ell_1'\ell_1 r/1.2, \ell_2'\ell_2 r/.8)} R(g_{r^{2-\eta}})dx = 0.
\end{equation*}
Since $R(g_t)\geq 0$ for all $t>0$ by Lemma \ref{lemma:RFNNSC}, and since $A(0, \ell_1 r, \ell_2 r) \subseteq A(0, \ell_1'\ell_1 r/1.2, \ell_2'\ell_2 r/.8)$, it follows that 
\begin{equation*}
\lim_{r\to\infty} \int_{A(0, \ell_1 r, \ell_2 r)} R(g_{r^{2-\eta}}) = 0.
\end{equation*}
\end{proof}

\subsection{Weighted H\"older spaces}
Here we record some results about operators between weighted H\"older spaces. The following is \cite[Corollary A.42]{Lee19}:
\begin{theorem}\label{thm:LeeFredholm}
Let $(M^n, g)$ be a complete, connected, smooth, $C^2$-asymptotically flat manifold (so $g$ is smooth here), such that $R(g)\in L^1(M)$. Assume $n\geq 3$, and let $s$ be any real number not in the exceptional set $\Z \setminus (2-n, 0)$. 

Define the following weighted H\"older norms:
\begin{equation*}
||u||_{C_s^{k,\alpha}(M)} := \sum_{i=0}^{k} \sup_{x\in M} |\ell(x)^{i-s}\nabla^i u| + \sup_{x\in M} \ell(x)^{k+\alpha - s}[\nabla^k u]_{C^\alpha(B(x, |x|/2))}
\end{equation*}
where $\ell(x)$ is a smooth positive function on $M$ that agrees with $|x|$ (in coordinates) outside of a compact set, and
\begin{equation*}
[v]_{C^\alpha(U)} := \sup_{x,y\in U, x\neq y}\frac{|v(x) - v(y)|}{d_g(x,y)^\alpha}.
\end{equation*}

Consider an operator $L u:= -\Delta^{g} u + \langle V, \nabla u\rangle + qu$, where $V$ is a smooth section of $TM$ such that $V \in C^{0,\alpha}_{-1-\gamma}(M)$ and $q$ is a smooth function on $M$ such that $q\in C^{0,\alpha}_{-2-\gamma}(M)$ for some $\alpha\in (0,1)$ and $\gamma>0$. Then $L: C^{2,\alpha}_{s}(M)\to C^{0,\alpha}_{s-2}(M)$ is a Fredholm operator whose index is the same as that of the Euclidean Laplacian $\Delta^\delta: C^{2,\alpha}_{s}(\R^n)\to C^{0,\alpha}_{s-2}(\R^n)$.
\end{theorem}

\begin{remark}\label{rmk:normsequivalent}
Note that as stated, the norms $||\cdot||_{C_s^{2,\alpha}(M)}$ in Theorem \ref{thm:LeeFredholm} depend on choice of asymptotically flat metric. However, if $g$ is a $C^2$-asymptotically flat metric on $\R^n$ that is uniformly $(1+\lambda)$-bilipschitz  to the Euclidean metric $\delta$ for some $\lambda < 1$, then there exist constants $c = c(\lambda, c_1, \alpha, s, n)$ and $C = C(\lambda, c_1, \alpha, s, n)$ depending only on $\lambda$, $\alpha$, $s$, $n$, and a constant $c_1$ satisfying
\begin{equation*}
\sup_{x\in \R^n} \ell(x)|\nabla^\delta g(x)| + \ell(x)^2|(\nabla^\delta)^2 g(x)| \leq c_1
\end{equation*}
such that for all $u$,
\begin{equation*}
c||u||_{C_s^{2,\alpha}(\R^n, \delta)} \leq || u||_{C_s^{2,\alpha}(\R^n, g)} \leq C || u||_{C_s^{2,\alpha}(\R^n, \delta)}.
\end{equation*}
In particular, the weighted norms $||\cdot||_{C_s^{2,\alpha}(\R^n)}$ with respect to $g$ and $\delta$ are equivalent.

To see why this is the case, first note that for any $x\neq 0$, there exists $N(n)$-many balls of the form $B_\delta(z, |x|/8)$, where $z\in B_\delta(x, 3|x|/4)$, that form a cover of $B_\delta(x, 3|x|/4)$. Then replacing each ball in the finite cover with $B_\delta(z, |z|/2)$ also yields a finite cover of $B_\delta(x, 3|x|/4)$ (since, if $z\in B_\delta(x, 3|x|/4)$ then $|z|/2 \geq |x|/8$ so this procedure only enlarges each ball in the cover). Then, because $\lambda < 1$, $\sqrt{1+\lambda}/2 < 3/4$ this is also a cover of $B(x, \sqrt{1+\lambda}|x|/2)$. Therefore, for any function $f: \R^n\times \R^n \to \R$ we have
\begin{equation}
\begin{split}
\sup_{x\in \R^n}|x|^{\alpha + 2-s}\sup_{y,w\in B_g(x, \sqrt{1+\lambda}|x|/2)} f(y,w) & \leq \sup_{x\in \R^n}\ell(x)^{\alpha + 2-s}\sup_{y,w\in B_\delta(x, \sqrt{1+\varepsilon}|x|/2)} f(y,w)
\\&\leq \sup_{x\in \R^n}\ell(x)^{\alpha + 2-s}\max\left\{\sup_{y,w\in B_\delta(z_i, |z_i|/2)} f(y, w): i = 1, 2, \ldots, N(n)\right\}
\\& \leq \sup_{x\in \R^n}c\ell(z_i)^{2+\alpha - s}\max\left\{\sup_{y,w\in B_\delta(z_i, |z_i|/2)} f(y, w): i = 1, 2, \ldots, N(n)\right\}
\\& \leq c\sup_{x\in \R^n}\ell(x)^{2+\alpha - s}\sup_{y,w \in B_\delta(|x|, |x|/2)}f(y,w)
\end{split}
\end{equation} 
where, note, the $z_i$ depend on $x$, and $c$ is either $(1/4)^{2+\alpha - s}$ or $(7/4)^{2+\alpha - s}$ depending on the sign of $2 + \alpha - s$.

We then have
\begin{align*}
||u||_{C_s^{2,\alpha}(\R^n, g)} &= \sup_{x\in \R^n} |\ell(x)^{-s}u(x)| + \sup_{x\in \R^n}|\ell(x)^{-s + 1}du|_{g} + \sup_{x\in \R^n}|\ell(x)^{2-s} (\nabla^g)^2u|_{g} 
\\& + \sup_{x\in \R^n}\ell(x)^{\alpha + 2-s}\sup_{y,z\in B_{g}(x, |x|/2), y\neq z}\frac{|(\nabla^g)^2 u(y) - (\nabla^g)^2 u(z) |}{d_{g}(y,z)^\alpha}
\\ & \leq \sup_{x\in \R^n} |\ell(x)^{-s}u(x)| + c(\lambda)\sup_{x\in \R^n}|\ell(x)^{-s + 1}du|_{\delta} + c(\lambda)\sup_{x\in \R^n}|\ell(x)^{2-s} (\nabla^{\delta})^2u|_{\delta} 
\\& + (1+\varepsilon)c(n,\varepsilon)\sup_{x\in \R^n}|\ell(x)^{2-s} |\nabla^{\delta}g| |\nabla^{\delta}u||_{\delta}
\\& + c(\lambda, \alpha)\sup_{x\in \R^n}|x|^{\alpha + 2-s}\sup_{y,z\in B_{\delta}(x, |x|/2), y\neq z}\frac{|(\nabla^{g})^2 u(y) - (\nabla^{g})^2 u(z) |}{d_{\delta}(y,z)^\alpha}
\\ & \leq \sup_{x\in \R^n} |\ell(x)^{-s}u(x)| + c(\lambda)\sup_{x\in \R^n}|\ell(x)^{-s + 1}du|_{\delta} + c(\lambda)\sup_{x\in \R^n}|\ell(x)^{2-s} (\nabla^{\delta})^2u|_{\delta} 
\\& + c(n,\lambda)\sup_{x\in \R^n}|\ell(x)^{2-s}\nabla^\delta g(x) * \nabla u(x)|
\\& + c(n,\lambda, \alpha) \sup_{x\in \R^n}\ell(x)^{2+\alpha - s}\sup_{y,z\in B_\delta(x, |x|/2), y\neq z} \frac{|(\nabla^\delta)^2u(y) - (\nabla^{\delta})^2u(z)|_\delta}{d_\delta^\alpha(y,z)}
\\& + c(n, \lambda, \alpha)\sup_{x\in \R^n}\ell(x)^{2+\alpha - s}\sup_{y,z\in B_\delta(x, |x|/2), y\neq z} \frac{|\nabla^\delta g*\nabla u(y) - \nabla^\delta g * \nabla u(z)|_\delta}{d_\delta^\alpha(y,z)}
\\& \leq c(\lambda)||u||_{C^{2,\alpha}_s(\R^n, \delta)} + c(n,\lambda, c_1)\sup_{x\in \R^n}|\ell(x)^{1-s}* \nabla u(x)|_\delta
\\& + c(n,\lambda,\alpha)\sup_{x\in \R^n}\ell(x)^{2+\alpha - s}\sup_{y,z\in B_\delta(x, |x|/2), y\neq z} \frac{|(\nabla^\delta)^2u(y) - (\nabla^{\delta})^2u(z)|_\delta}{d_\delta^\alpha(y,z)}
\\& + c(n, \lambda, \alpha)\sup_{x\in \R^n}\ell(x)^{2+\alpha - s}\sup_{y,z\in B_\delta(x, |x|/2), y\neq z} \frac{|\nabla^\delta g*\nabla u(y) - \nabla^\delta g * \nabla u(z)|_\delta}{d_\delta^\alpha(y,z)}
\\& \leq c(n,\lambda, \alpha, c_1)||u||_{C^{2,\alpha}_s(\R^n, \delta)} + c(n,\lambda, \alpha, c_1)\ell(x)^{1+\alpha - s}\sup_{y,z\in B_\delta(x, |x|/2), y\neq z} \frac{|\nabla u(y) - \nabla u(z)|_\delta}{d_\delta^\alpha(y,z)}
\\& \leq c(n,\lambda, \alpha, c_1)||u||_{C^{2,\alpha}_s(\R^n, \delta)} 
\\& + c(n,\lambda, \alpha, c_1)\ell(x)^{1+\alpha - s}\sup_{y,z\in B_\delta(x, |x|/2), y\neq z}|| (\nabla^\delta)^2 u||_{L^\infty(B_\delta(x, |x|/2))}d_\delta(y,z)^{1-\alpha}
\\& \leq c(n,\lambda, \alpha, c_1)||u||_{C^{2,\alpha}_s(\R^n, \delta)} 
\\& + c(n,\lambda, \alpha, c_1)\ell(x)^{2 - s}\sup_{y,z\in B_\delta(x, |x|/2), y\neq z}|| (\nabla^\delta)^2 u||_{L^\infty(B_\delta(x, |x|/2))}
\\& \leq c(n,\lambda, \alpha, c_1)||u||_{C^{2,\alpha}_s(\R^n, \delta)}.
\end{align*}
The other inequality is similar.
\end{remark}

We also use the following standard result:
\begin{lemma}\label{lemma:isomorphismthreshold}
Let $T: X\to Y$ be an isomorphism between Banach spaces. A bounded linear operator $S:X\to Y$ is also an isomorphism if $||S - T||_{\op} < \inf_{||u||_X=1}||Tu||_{Y} =: \mod(T)$.
\end{lemma}

\section{Proof of Theorem \ref{thm:C0PMT}}\label{sec:approximationproof} 
The goal of this section is to prove Theorem \ref{thm:masssequence}. The following are fixed throughout this section:
\begin{itemize}
\item $\alpha \in (0,1)$ some fixed H\"older exponent.
\item $n\geq 3$ the dimension.
\item $\tau \in (\tfrac{n-2}{2}, n-2)$.
\item $\eta \in (0, \min\{ \tfrac{2\tau - n + 2}{2}, 1, \tfrac{2}{3}\tau\})$.
\item $\bar \varepsilon = \bar \varepsilon(n, \alpha, \tau)>0$, a positive number that is less than or equal to $\min \{\varepsilon, \varepsilon', \varepsilon'', \varepsilon'''\}$ where $\varepsilon $ is as in Lemma \ref{lemma:KL}, $\varepsilon' = \varepsilon'(n,\tau)$ is as in Lemma \ref{lemma:RDTFC0AF}, $\varepsilon'' = \varepsilon''(n,\tau)$ is as in Corollary \ref{cor:smoothdecay}, and $\varepsilon''' = \varepsilon'''$ is as in Lemma \ref{lemma:RDTFL2distortion}. We specify $\bar \varepsilon$ in the proof of Lemma \ref{lemma:conformaldeformation}.
\item $g$ a $C^0$-asymptotically flat metric on $\R^n$ with decay rate $\tau$, such that $|| g - \delta||_{L^\infty(\R^n)} < \bar \varepsilon$.
\item $(g_t)_{t>0}$ a Ricci-DeTurck flow starting from $g$ in the sense of Lemma \ref{lemma:KL}. Because of the conditions on $\bar \varepsilon$, the estimates of Lemma \ref{lemma:RDTFC0AF}, Corollary \ref{cor:smoothdecay}, and Lemma \ref{lemma:massdistortionestimate} hold. Moreover, by Lemma \ref{lemma:KL}, for all $t>0$, $|| g_t - \delta||_{L^\infty(\R^n)} \leq c_n \bar \varepsilon$.
\item $\varphi$ a smooth cutoff function $\R \to \R^{\geq 0}$ with $\supp(\varphi)\subset\subset (.9, 1.1)$.
\item $T = T(n, \tau, \eta, \alpha, \supp(\varphi), || \rho^\tau (g - \delta)||, \ell) >1$, where $\rho$ is as in Lemma \ref{lemma:RDTFC0AF} and $\ell$ is as in Theorem \ref{thm:LeeFredholm}, such that $\sqrt{T} \geq \bar r$, where $\bar r$ is the quantity from Lemma \ref{lemma:massdistortionestimate}, and $\ell(x) = |x|$ for all $|x| \geq \sqrt{T}/2$.
\item $\chi$ a nonnegative smooth cutoff function that is identically equal to $1$ on $B(0,1/2)$ and identically equal to 0 outside of $B(0,2)$, such that $|\nabla \chi| \leq 10$, and, for $t> 1$, $\chi_t(y) := \chi(y/\sqrt{t})$ so that $\chi_t\equiv 1$ on $B(0, \sqrt{t}/2)$ and $\chi_t \equiv 0$ outside of $B(0, 2\sqrt{t})$, and for $k= 1,2,3,4$, $|\nabla^k \chi_t| \leq 10/t^{k/2}$, and
\item for $t>1$, $\hat g_t := \chi_t g_{t^{1-\eta/2}} + (1-\chi_t)\delta$, so that $\hat g_t$ is identically equal to $g_{t^{1-\eta/2}}$ on $B(0, \sqrt{t}/2)$ and identically Euclidean outside of $B(0, 2\sqrt{t})$.
\end{itemize}
\begin{remark}\label{rmk:gluedderivests}
By Corollary \ref{cor:smoothdecay} and the derivative estimates for $\chi_t$, we have that for $|x| > 1$ and $k = 0, 1, 2, 3, 4$
\begin{equation*}
|\nabla^k (\hat g_t - \delta)(x)| \leq \frac{c_{k,n}|x|^{-\tau}}{(t^{1-\eta/2})^{k/2}}
\end{equation*}
and also
\begin{equation*}
||\nabla^k \hat g_t||_{L^\infty(\R^n)} \leq \frac{c_{k,n}|| g - \delta||_{L^\infty(\R^n)} }{(t^{1-\eta/2})^{k/2}} < \frac{C_{k,n}\bar\varepsilon}{(t^{1-\eta/2})^{k/2}}.
\end{equation*}
\end{remark}

Throughout the rest of this section we let $c$ be a constant whose value changes from line to line, such that $c$ does not depend on $t$ or $r$ (eventually we let $t = r^2 \to \infty$).
\begin{lemma}\label{lemma:conformaldeformation}
For $t \geq T = T(n, \tau, \eta, \alpha, ||\rho^\tau (g - \delta)||_{L^\infty(\R^n)}, \ell)$ there exists a smooth positive solution $u^t$ to $L^{\hat g_t}u^t = \chi_tR(g_{t^{1-\eta/2}})$, where $L^{\hat g_t} := -\frac{4(n-1)}{n-2}\Delta^{\hat g_t} + R(\hat g_t)$ is the conformal Laplacian of $\hat g_t$, such that $u^t - 1\in C^{2,\alpha}_{-\tau}(\R^n)$. In particular, the metric $g_t' = (u^t)^{4/(n-2)}\hat g_t$ is $C^2$-asymptotically flat with decay rate $\tau$, with scalar curvature equal to $\chi_tR(g_{t^{1-\eta/2}})\geq 0$.
\end{lemma}
We remark that this choice of conformal factor is used by Dan Lee in \cite{Lee19-errata}
\begin{proof}
Fix some $t>0$, which we will later select to be large, and some $s\in (2-n, 0)$, which we will later set to $-\tau$. We first show that the operator $\tilde L := -\frac{4(n-1)}{n-2}\Delta^{\hat g_t} + \chi_tR(g_{t^{1-\eta/2}})$ is an isomorphism $C^{2,\alpha}_{s}\to C^{0,\alpha}_{s-2}$ and then we show that $L^{\hat g_t}$ is sufficiently close to $\tilde L$ so that it too is an isomorphism. To do this we apply Theorem \ref{thm:LeeFredholm} to $\tilde L$: First note that $\hat g_t$ is complete and classically asymptotically flat, since it is in fact identically Euclidean outside of a compact set. Then, since $\chi_t$ and hence $\chi_t R(g_{t^{1-\eta/2}})$ are identically $0$ outside of a compact set, we have $\chi_t R(g_{t^{1-\eta}/2}) \in C^{0,\alpha}_{s-2}$ for any $\alpha\in (0,1)$.

In particular, we may apply Theorem \ref{thm:LeeFredholm} to $\frac{n-2}{4(n-1)}\tilde L = -\Delta^{\hat g_t} + \frac{n-2}{4(n-1)}\chi_t R(g_{t^{1-\eta/2}})$ to show that it is Fredholm and has the same index as that of the Euclidean Laplacian $\Delta^{\delta}:C_s^{2,\alpha}\to C_{s-2}^{0,\alpha}$. Since $s\in (2-n, 0)$, by \cite[Theorem A.35]{Lee19}, $\Delta^\delta$ is an isomorphism, and hence has Fredholm index equal to $0$, so $\tilde L$ has Fredholm index equal to 0 as well. Therefore $\tilde L$ is injective if and only if it is surjective. We now show that $\tilde L$ is injective and hence is an isomorphism:  since $R(g_t)\geq 0$ and hence $\chi_tR(g_{t^{1-\eta/2}})\geq 0$, and since $\R^n$ is connected, we may apply the Maximum Principle \cite[Theorem A.2]{Lee19} to $\tfrac{n-2}{4(n-1)} \tilde L= -\Delta^{\hat g_t} + \tfrac{n-2}{4(n-1)}\chi_t R(g_{t^{1-\eta/2}})$ to find that it is injective, and hence, since it has Fredholm index equal to 0, it is also surjective, and is an isomorphism.

It remains to show that $L^{\hat g_t}$ is sufficiently close to $\tilde L$ to also be an isomorphism. Of course, since being an isomorphism is an open condition, there exists some $\gamma >0$ such that if $||L^{\hat g_t} - \tilde L||_{\op} < \gamma$ then $L^{\hat g_t}$ is an isomorphism. The problem is that $\tilde L$ depends on $t$, so a priori $\gamma$ might depend on $t$ as well. We show that by choosing $t$ large enough depending on $n, \tau, \eta, s, \ell$, and $||\rho^\tau(g - \delta)||_{L^\infty(\R^n)}$, one can ensure that $\gamma$ is independent of $t$. To do this, note that by Lemma \ref{lemma:isomorphismthreshold} it is sufficient to show that, by $t$ large enough, we have $||L^{\hat g_t} - \tilde L||_{\op} < \mod(\tilde L)$. We first show that, for large $t$, $-c(n)\tilde L = -\frac{n-2}{4(n-1)}\tilde L$ is sufficiently close to $\Delta^{\delta}$ so that $\mod(\tilde L)$ is bounded below by some number independent of $t$.  To do this, fix some $u\in C^{2,\alpha}_{s}(\R^n)$. Then, using Remark \ref{rmk:gluedderivests} and the definition of $\chi_t$, we have the following for any $1 < r_0 < 2\sqrt{t}$, where $r_0$ is chosen sufficiently large so that $\ell(x) = |x|$ for all $|x| \geq r_0/2$:
\begin{align*}
&\sup_{x\in \R^n}|\ell(x)^{-s + 2}(-c(n)\tilde L  u - \Delta^{\delta} u)|
\\& \leq c\sup_{x\in B(0, 2\sqrt{t})}\ell(x)^{-s + 2} \left(|\hat g_t - \delta||\nabla^2 u(x)| + |\nabla(\hat g_t - \delta)||\nabla u(x)| + |\nabla^2(\hat g_t - \delta)||u(x)|\right)
%
\\& \leq \max\bigg\{c\sup_{x\in B(0, r_0)}\ell(x)^{-s + 2} \left(|\hat g_t - \delta||\nabla^2 u(x)| + |\nabla(\hat g_t - \delta)||\nabla u(x)| + |\nabla^2(\hat g_t - \delta)||u(x)|\right),
\\& \qquad \qquad \qquad \qquad c\sup_{x\in A(0, r_0, 2\sqrt{t})}\ell(x)^{-s + 2} \left(|\hat g_t - \delta||\nabla^2 u(x)| + |\nabla(\hat g_t - \delta)||\nabla u(x)| + |\nabla^2(\hat g_t - \delta)||u(x)|\right) \bigg\}
\\& \leq \max\bigg\{c\sup_{x\in B(0, r_0)}||u||_{C^{2,\alpha}_s(\R^n)}(|\hat g_t - \delta| + \ell(x)|\nabla(\hat g_t - \delta)| + \ell(x)^2|\nabla^2(\hat g_t - \delta)|),
\\& \qquad \qquad \qquad \qquad c\sup_{x\in A(0, r_0, 2\sqrt{t})}||u||_{C^{2,\alpha}_s(\R^n)}(|\hat g_t - \delta|(x) + \ell(x)|\nabla(\hat g_t - \delta)|(x) + \ell^2(x)|\nabla^2(\hat g_t - \delta)|(x))\bigg\}
\\& \leq \max\bigg\{c\sup_{x\in B(0, r_0)}||u||_{C^{2,\alpha}_s(\R^n)}(|| g_0 - \delta||_{L^\infty(\R^n)} + r_0\frac{|| g_0 - \delta||_{L^\infty(\R^n)}}{\sqrt{t^{1-\eta/2}}} +  r_0^2\frac{|| g_0 - \delta||_{L^\infty(\R^n)}}{t^{1-\eta/2}}),
\\& \qquad \qquad \qquad \qquad c\sup_{x\in A(0, r_0, 2\sqrt{t})}||u||_{C^{2,\alpha}_s(\R^n)}(|x|^{-\tau} + |x|\frac{|x|^{-\tau}}{\sqrt{t^{1-\eta/2}}} + |x|^2\frac{|x|^{-\tau}}{t^{1-\eta/2}})\bigg\}
\\& \leq c\max\bigg\{(|| g_0 - \delta||_{L^\infty(\R^n)} + r_0 \frac{|| g_0 - \delta||_{L^\infty(\R^n)}}{\sqrt{t^{1-\eta/2}}} + r_0^2\frac{|| g_0 - \delta||_{L^\infty(\R^n)}}{t^{1-\eta/2}}),
\\& \qquad \qquad \qquad \qquad c\sup_{x\in A(0, r_0, 2\sqrt{t})} |x|^{-\tau} + |x|^{1 -\eta/2}\frac{|x|^{-\tau + \eta/2}}{\sqrt{t^{1-\eta/2}}} + |x|^{2-\eta}\frac{|x|^{-\tau + \eta}}{t^{1-\eta/2}})\bigg\}||u||_{C^{2,\alpha}_s(\R^n)}
\\& \leq c\max\bigg\{(|| g_0 - \delta||_{L^\infty(\R^n)}+  r_0\frac{|| g_0 - \delta||_{L^\infty(\R^n)}}{\sqrt{t^{1-\eta/2}}} +  r_0^2\frac{|| g_0 - \delta||_{L^\infty(\R^n)}}{t^{1-\eta/2}}),
\\& \qquad \qquad \qquad \qquad c\sup_{x\in A(0, 1, 2\sqrt{t})} |x|^{-\tau} + |x|^{1 -\eta/2}\frac{|x|^{-\tau + \eta/2}}{\sqrt{t^{1-\eta/2}}} + |x|^{2-\eta}\frac{|x|^{-\tau + \eta}}{t^{1-\eta/2}})\bigg\}||u||_{C^{2,\alpha}_s(\R^n)}
\\& \leq c\max\bigg\{(|| g_0 - \delta||_{L^\infty(\R^n)} + \frac{|| g_0 - \delta||_{L^\infty(\R^n)}}{\sqrt{t^{1-\eta/2}}} + r_0^2 \frac{|| g_0 - \delta||_{L^\infty(\R^n)}}{t^{1-\eta/2}}), c r_0^{-\tau} + r_0^{-\tau + \eta/2} + r_0^{-\tau + \eta}\bigg\}||u||_{C^{2,\alpha}_s(\R^n)}
\\& \leq 2\bar \varepsilon ||u||_{C^{2,\alpha}_s(\R^n)}
\end{align*}
for $r_0$ sufficiently large depending on $\bar \varepsilon$, $\tau$, and $\eta$ and for $t$ sufficiently large depending on $\bar \varepsilon$, $r_0$, and $\eta$, where we are using the fact that $\eta < 2\tau/3$ (first choose $r_0$ large to bound the second term, then choose $t$ large depending on $r_0$ to bound the first).

We also have
\begin{align*}
& \sup_{x\in \R^n} \ell(x)^{\alpha -s + 2}\sup_{y,z\in B(x, |x|/2), y\neq z}\frac{|-c(n)\tilde L  u(y) - \Delta^{\delta} u(y) - (-c(n) \tilde L u(z) - \Delta^{\delta} u(z))|}{d_{\delta}(y,z)^{\alpha}}
\\& \leq  \sup_{x\in B(0, 4\sqrt{t})}\ell(x)^{\alpha -s + 2}\sup_{y,z\in B(x, |x|/2), y\neq z}\frac{| (\hat g_t - \delta)(y)* \nabla^2 u(y) - (\hat g_t - \delta)(z)* \nabla^2 u(z)|}{d_{\delta}(y,z)^{\alpha}}
\\& + \sup_{x\in B(0, 4\sqrt{t})}\ell(x)^{\alpha -s + 2}\sup_{y,z\in B(x, |x|/2), y\neq z}\frac{| \nabla(\hat g_t - \delta)(y)* \nabla u(y) - \nabla(\hat g_t - \delta)(z)* \nabla u(z)|}{d_{\delta}(y,z)^{\alpha}}
\\& + \sup_{x\in B(0, 4\sqrt{t})}\ell(x)^{\alpha -s + 2}\sup_{y,z\in B(x, |x|/2), y\neq z}\frac{| \chi_t R(g_{t^{1-\eta}})(y) u(y) -   \chi_t R(g_{t^{1-\eta}})(z) u(z)|}{d_{\delta}(y,z)^{\alpha}}
\\& \leq \sup_{x\in B(0, 4\sqrt{t})}\ell(x)^{\alpha -s + 2}\sup_{y,z\in B(x, |x|/2), y\neq z}\bigg[\frac{| (\hat g_t - \delta)(y)* \nabla^2 u(y) - (\hat g_t - \delta)(y)* \nabla^2 u(z)|}{d_{\delta}(y,z)^{\alpha}}
\\& \qquad \qquad  \qquad \qquad + \frac{| (\hat g_t - \delta)(y)* \nabla^2 u(z) - (\hat g_t - \delta)(z)* \nabla^2 u(z)|}{d_{\delta}(y,z)^{\alpha}}\bigg]
\\& + \sup_{x\in B(0, 4\sqrt{t})}\ell(x)^{\alpha -s + 2}\sup_{y,z\in B(x, |x|/2), y\neq z}\bigg[\frac{| \nabla(\hat g_t - \delta)(y)* \nabla u(y) - \nabla(\hat g_t - \delta)(y)* \nabla u(z)|}{d_{\delta}(y,z)^{\alpha}}
\\& \qquad \qquad \qquad \qquad + \frac{| \nabla(\hat g_t - \delta)(y)* \nabla u(z) - \nabla(\hat g_t - \delta)(z)* \nabla u(z)|}{d_{\delta}(y,z)^{\alpha}} \bigg]
\\& + \sup_{x\in B(0, 4\sqrt{t})}\ell(x)^{\alpha -s + 2}\sup_{y,z\in B(x, |x|/2), y\neq z}\bigg[\frac{| \chi_t R(g_{t^{1-\eta}})(y) u(y) -   \chi_t R(g_{t^{1-\eta}})(y) u(z)|}{d_{\delta}(y,z)^{\alpha}}
\\& \qquad \qquad \qquad \qquad + \frac{\chi_t R(g_{t^{1-\eta}})(y) u(z) -   \chi_t R( g_{t^{1-\eta}})(z) u(z)}{d_{\delta}(y,z)^{\alpha}}\bigg]
=: I + II + III.
\end{align*}
\begin{align*}
I & \leq c||\hat g_t - \delta||_{C^0(\R^n)} ||u||_{C^{2,\alpha}_{s}(\R^n)} + \sup_{x\in B(0, 4\sqrt{t})}\ell(x)^{\alpha}\sup_{y,z\in B(x, |x|/2), y\neq z} ||u||_{C^{2,\alpha}_{s}(\R^n)}\frac{| (\hat g_t - \delta)(y)- (\hat g_t - \delta)(z)|}{d_{\delta}(y,z)^{\alpha}}
\\& \leq c||\hat g_t - \delta||_{C^0(\R^n)} ||u||_{C^{2,\alpha}_{s}(\R^n)} + \sup_{x\in B(0, 4\sqrt{t})} ||\nabla (\hat g_t - \delta) ||_{C^0(B(x, |x|/2))} |x|||u||_{C^{2,\alpha}_{s}(\R^n)}
\\& \leq c||\hat g_t - \delta||_{C^0(\R^n)} ||u||_{C^{2,\alpha}_{s}(\R^n)} 
\\& \qquad \qquad + \max\bigg\{\sup_{x\in B(0, r_0)} ||\nabla (\hat g_t - \delta) ||_{C^0(B(x, |x|/2))} |x|, \sup_{x\in A(0, r_0, 4\sqrt{t})} ||\nabla (\hat g_t - \delta) ||_{C^0(B(x, |x|/2))} |x|  \bigg\} ||u||_{C^{2,\alpha}_{s}(\R^n)}
\\& \leq c|| g_0 - \delta||_{L^\infty(\R^n)} ||u||_{C^{2,\alpha}_{s}(\R^n)} + \max\bigg\{ \frac{cr_0|| g_0 - \delta||_{L^\infty(\R^n)}}{\sqrt{t^{1-\eta/2}}}, \sup_{x\in A(0, r_0, 4\sqrt{t})}\frac{c|x|^{1 - \eta/2}|x|^{-\tau + \eta/2}}{\sqrt{t^{1-\eta/2}}} \bigg\} ||u||_{C^{2,\alpha}_{s}(\R^n)}
\\& \leq c|| g_0 - \delta||_{L^\infty(\R^n)} ||u||_{C^{2,\alpha}_{s}(\R^n)} + \max\bigg\{ \frac{cr_0|| g_0 - \delta||_{L^\infty(\R^n)}}{\sqrt{t^{1-\eta/2}}}, cr_0^{-\tau + \eta/2} \bigg\} ||u||_{C^{2,\alpha}_{s}(\R^n)}
\\& \leq 2c\bar \varepsilon||u||_{C^{2,\alpha}_{s}(\R^n)} 
\end{align*}
for $r_0$ sufficiently large depending on $\bar \varepsilon$, $\tau$, and $\eta$ and $t$ sufficiently large depending on $\bar \varepsilon, r_0$, and $\eta$.
Arguing similarly, we have
\begin{align*}
II & \leq \sup_{x\in B(0, 4\sqrt{t})}\ell(x)^{\alpha -s + 2}\sup_{y,z\in B(x, |x|/2), y\neq z} ||\nabla (\hat g_t - \delta)||_{C^0(B(x, |x|/2))}||\nabla^2 u||_{C^0(B(x, |x|/2))}d_{\delta}^{1-\alpha}(y,z)
\\& \qquad \qquad \qquad + \sup_{x\in B(0, 4\sqrt{t})}\ell(x)^{\alpha + 1}\sup_{y,z\in B(x, |x|/2), y\neq z}\frac{|\nabla(\hat g_t - \delta)(y) - \nabla(\hat g_t - \delta)(z)|}{d_\delta^\alpha(y,z)}||u||_{C^{2,\alpha}_s(\R^n)}
\\& \leq \sup_{x\in B(0, 4\sqrt{t})}|| \nabla(\hat g_t -\delta)||_{C^0(B(x, |x|/2))} \sup_{y,z\in B(x, |x|/2), y\neq z} d_\delta(y,z) ||u||_{C^{2,\alpha}_s(\R^n)}
\\& \qquad \qquad \qquad + \sup_{x\in B(0, 4\sqrt{t})}\ell(x)^{\alpha + 1}\sup_{y,z\in B(x, |x|/2), y\neq z} || \nabla^2(\hat g_t - \delta) ||_{C^0(B(x, |x|/2))}d^{1-\alpha}_{\delta}(y,z)||u||_{C^{2,\alpha}_s(\R^n)}
\\& \leq \max\bigg\{ \sup_{x\in B(0, r_0)} \frac{c r_0|| g_0 - \delta||_{L^\infty(\R^n)}}{\sqrt{t^{1-\eta/2}}}, \sup_{x\in A(0, r_0, 4\sqrt{t})}\frac{|x|^{1-\eta/2}|x|^{-\tau + \eta/2}}{\sqrt{t^{1-\eta/2}}} \bigg\}||u||_{C^{2,\alpha}_s(\R^n)}
\\& \qquad \qquad \qquad + \max\bigg\{  \sup_{x\in B(0, r_0)} \frac{cr_0^2|| g_0 - \delta||_{L^\infty(\R^n)}}{t^{1-\eta/2}}, \sup_{x\in A(0, r_0, 4\sqrt{t})} \frac{|x|^{2-\eta}|x|^{-\tau + \eta}}{t^{1-\eta/2}} \bigg\}||u||_{C^{2,\alpha}_s(\R^n)}
\\& \leq \max\bigg\{ \sup_{x\in B(0, r_0)} \frac{c r_0|| g_0 - \delta||_{L^\infty(\R^n)}}{\sqrt{t^{1-\eta/2}}}, cr_0^{-\tau + \eta/2} \bigg\}||u||_{C^{2,\alpha}_s(\R^n)}
 \\& \qquad \qquad  + \max\bigg\{  \sup_{x\in B(0, r_0)} \frac{cr_0^2|| g_0 - \delta||_{L^\infty(\R^n)}}{t^{1-\eta/2}}, cr_0^{-\tau + \eta} \bigg\}||u||_{C^{2,\alpha}_s(\R^n)}
\\& \leq 2c\bar \varepsilon||u||_{C^{2,\alpha}_{s}(\R^n)} 
\end{align*}
(for $r_0$ sufficiently large depending on $\bar \varepsilon, \tau,$ and $\eta$, and for $t$ sufficiently large depending on $\bar \varepsilon$, $r_0$, and $\eta$)
and
\begin{align*}
III & \leq \sup_{x\in B(0, 4\sqrt{t})}\ell(x)^{\alpha -s + 2}\sup_{y,z\in B(x, |x|/2), y\neq z}\bigg[ ||\chi_t R(g_{t^{1-\eta}})||_{C^0(B(x, |x|/2))} || \nabla u||_{C^0(B(x, |x|/2))}d_{\delta}^{1-\alpha}(y,z)
\\& \qquad \qquad \qquad \qquad + || \nabla(\chi_t R(g_{t^{1-\eta/2}}))||_{C^0(B(x, |x|/2))}d_\delta^{1-\alpha}(y,z)||u||_{C^0(B(x, |x|/2))}\bigg]
\\& \leq \sup_{x\in B(0, 4\sqrt{t})}\bigg[ \ell(x)^{\alpha + 1} ||\chi_t R(g_{t^{1-\eta/2}})||_{C^0(B(x, |x|/2))}|x|^{1-\alpha}
\\& \qquad \qquad \qquad \qquad + \ell(x)^{\alpha+ 2} || \nabla(\chi_t R(g_{t^{1-\eta/2}}))||_{C^0(B(x, |x|/2))}|x|^{1-\alpha}\bigg]||u||_{C^{2,\alpha}_s(\R^n)}
\\& \leq \max\bigg\{ \sup_{x\in B(0, r_0)}\ell(x)^{\alpha + 1} ||\chi_t R(g_{t^{1-\eta/2}})||_{C^0(B(x, |x|/2))}|x|^{1-\alpha},
\\& \qquad \qquad  \sup_{x\in A(0, r_0, 4\sqrt{t})}\ell(x)^{\alpha + 1} ||\chi_t R(g_{t^{1-\eta/2}})||_{C^0(B(x, |x|/2))}|x|^{1-\alpha}\bigg\}||u||_{C^{2,\alpha}_s(\R^n)}
\\& \qquad \qquad \qquad \qquad + \max\bigg\{\sup_{x\in B(0, r_0)}\ell(x)^{\alpha+ 2} || \nabla(\chi_t R(g_{t^{1-\eta/2}}))||_{C^0(B(x, |x|/2))}|x|^{1-\alpha},
\\& \qquad \qquad \qquad \qquad \qquad  \sup_{x\in A(0, r_0, 4\sqrt{t})}\ell(x)^{\alpha+ 2} || \nabla(\chi_t R(g_{t^{1-\eta/2}}))||_{C^0(B(x, |x|/2))}|x|^{1-\alpha}\bigg\}||u||_{C^{2,\alpha}_s(\R^n)}
\\& \leq \max\bigg\{c r_0^2 \frac{|| g_0 - \delta||_{L^\infty(\R^n)}}{t^{1-\eta/2}} , \sup_{x\in A(0, r_0, 4\sqrt{t})} \frac{c|x|^{2 - \eta}|x|^{-\tau + \eta}}{t^{1-\eta/2}}\bigg\}||u||_{C^{2,\alpha}_s(\R^n)} 
\\& \qquad \qquad \qquad \qquad + \max\bigg\{ r_0^3\frac{|| g_0 - \delta||_{L^\infty(\R^n)}}{\sqrt{t^{1-\eta/2}}^3}, \sup_{x\in A(0, r_0, 4\sqrt{t})}\frac{|x|^{3 - 3\eta/2}|x|^{-\tau + 3\eta/2}}{\sqrt{t^{1-\eta/2}}^3}\bigg\}||u||_{C^{2,\alpha}_s(\R^n)}
\\& \leq \max\bigg\{c r_0^2 \frac{|| g_0 - \delta||_{L^\infty(\R^n)}}{t^{1-\eta/2}} , cr_0^{-\tau + \eta}\bigg\}||u||_{C^{2,\alpha}_s(\R^n)}  + \max\bigg\{ r_0^3\frac{|| g_0 - \delta||_{L^\infty(\R^n)}}{\sqrt{t^{1-\eta/2}}^3}, cr_0^{-\tau + 3\eta/2}\bigg\}||u||_{C^{2,\alpha}_s(\R^n)}
\\& \leq 2c\varepsilon||u||_{C^{2,\alpha}_{s}(\R^n)} 
\end{align*}
for $r_0$ sufficiently large depending on $\bar \varepsilon, \tau,$ and $\eta$, and for $t$ sufficiently large depending on $\bar \varepsilon$, $r_0$, and $\eta$. 

Combining the previous four calculations, we have
\begin{equation*}\label{eq:tildeLclosetoDelta}
\begin{split}
|| -c(n)\tilde L  u - \Delta^{\delta} u||_{C^{0,\alpha}_{s-2}(\R)}  & \leq \sup_{x\in \R^n}|\ell(x)^{-s + 2}(-c(n)\tilde L  u - \Delta^{\delta} u)| + I + II + III
\\& \leq 8c\bar \varepsilon||u||_{C^{2,\alpha}_s(\R^n)}
\end{split}
\end{equation*}
for $t$ sufficiently large depending on $\bar \varepsilon, \tau,$ and $\eta$.

Now note that since $2-n < s < 0$, $\Delta^{\delta}: C^{2,\alpha}_s(\R^n)\to C^{0,\alpha}_{s-2}$ is an isomorphism \cite[Theorem A.35]{Lee19}, and hence $\mod(\Delta^{\delta}) >0$. Therefore, by (\ref{eq:tildeLclosetoDelta})
\begin{equation}\label{eq:minmoduniformlowerbound}
\begin{split}
\mod(\tilde L) &= \inf_{||u||_{C^{2,\alpha}_{s}} = 1}||\tilde L u||_{C^{0,\alpha}_{s-2}} 
\\& \geq \inf_{||u||_{C^{2,\alpha}_{s}} = 1}(||\Delta^{\delta} u||_{C^{0,\alpha}_{s-2}} - || \tilde L u - \Delta^{\delta} u ||_{C^{0,\alpha}_{s-2}} )
\\& \geq \mod(\Delta^{\delta}) - 8c\bar \varepsilon
\end{split}
\end{equation}
In particular, since $c$ does not depend on $t$,  this tells us that by choosing $\varepsilon$ sufficiently small depending on $\mod(\Delta^{\delta})$ (so depending on $n, s, \alpha$), we have that $\mod(\tilde L)$ is bounded below by some positive constant that is uniform as $t\to \infty$. 

We now show that for $t$ sufficiently large, we have $|| L^{\hat g_t} - \tilde L||_{\op} < \mod(\tilde L)$. 
\begin{claim}\label{claim:tildeLclosetoL}
For any $u\in C^{2,\alpha}_s(\R^n)$,
\begin{equation*}
 ||L^{\hat g_t}u - \tilde L u||_{C^{0,\alpha}_{s-2}(\R^n)} \leq c \sqrt{t}^{-\tau + 3\eta/2}||u||_{C^{2,\alpha}_s(\R^n)}.
\end{equation*}
Recall that, by choice of $\eta$, $\sqrt{t}^{-\tau + 3\eta/2}\xrightarrow[t\to\infty]{} 0$.
\end{claim}
\begin{proof}[Proof of Claim]
To see this, again fix $u\in C^{2,\alpha}_s(\R^n)$. Then, by definition of $\chi_t$ and $\hat g_t$,
\begin{align*}
& ||L^{\hat g_t}u - \tilde L u||_{C^{0,\alpha}_{s-2}(\R^n)} = \sup_{x\in \R^n}|\ell(x)^{-s + 2}( R(\hat g_{t})u - \chi_tR(g_{t^{1-\eta/2}})u)(x) | 
\\& + \sup_{x\in \R^n}\ell(x)^{\alpha -s + 2}\sup_{y,z\in B(x, |x|/2), y\neq z}\frac{|R(\hat g_{t})u(y) -  \chi_tR(g_{t^{1-\eta/2}})u(y) - R(\hat g_{t})u(z) +  \chi_tR(g_{t^{1-\eta/2}})u(z)|}{d_{\delta}(y,z)^\alpha}
\\&= \sup_{x\in A(0, .5\sqrt{t}, 2\sqrt{t})}|\ell(x)^{-s + 2}(R(\hat g_{t})u - \chi_tR(g_{t^{1-\eta/2}})u)(x) |
\\& + \sup_{x\in A(0, .5\sqrt{t}, 2\sqrt{t})}\ell(x)^{\alpha -s + 2}\sup_{y,z\in B(x, |x|/2), y\neq z}\frac{|(R(\hat g_{t}) -  \chi_tR(g_{t^{1-\eta/2}}))(y)(u(y) - u(z))|}{d_{\delta}(y,z)^\alpha}
\\& + \sup_{x\in A(0, .5\sqrt{t}, 2\sqrt{t})}\ell(x)^{\alpha -s + 2} \sup_{y,z\in B(x, |x|/2), y\neq z}\frac{|(R(\hat g_{t})(y) -  \chi_tR(g_{t^{1-\eta/2}}(y) - R(\hat g_{t})(z) +  \chi_tR(g_{t^{1-\eta/2}}(z))u(z)|}{d_{\delta}(y,z)^\alpha} 
\\& \leq c\sup_{x\in A(0, .5\sqrt{t}, 2\sqrt{t})} \ell(x)^{2}|R(\hat g_{t})u - \chi_tR(g_{t^{1-\eta/2}})| ||u||_{C^{2,\alpha}_s(\R^n)}
\\& + \sup_{x\in A(0, .5\sqrt{t}, 2\sqrt{t})}\ell(x)^{\alpha -s + 2}\sup_{y,z\in B(x, |x|/2), y\neq z}|(R(\hat g_{t}) -  \chi_tR(g_{t^{1-\eta/2}}))(y)|  ||\nabla u||_{C^{0}(B(x, |x|/2))}d^{1-\alpha}_\delta(y,z)
\\& + c\sup_{x\in A(0, .5\sqrt{t}, 2\sqrt{t})}\ell(x)^{\alpha -s + 2 } ||u||_{C^0(B(x, |x|/2))}(|| \nabla R(\hat g_t)||_{C^0(B(x, |x|/2))} 
\\& + ||\nabla(\chi_t R(g_{t^{1-\eta/2}}))||_{C^0(B(x, |x|/2))})d^{1-\alpha}_\delta(y,z)
\\& \leq c\sup_{x\in A(0, .5\sqrt{t}, 2\sqrt{t})} |x|^{2}\frac{|x|^{-\tau}}{t^{1-\eta}} ||u||_{C^{2,\alpha}_s(\R^n)} 
\\& \qquad \qquad + c\sup_{x\in A(0, .5\sqrt{t}, 2\sqrt{t})} |x|^{\alpha + 1}\frac{|x|^{-\tau}}{t^{1-\eta/2}}|x|^{1-\alpha} ||u||_{C^{2,\alpha}_s(\R^n)}
\\& \qquad \qquad \qquad \qquad +  c\sup_{x\in A(0, .5\sqrt{t}, 2\sqrt{t})}|x|^{\alpha + 2 }\frac{|x|^{-\tau}}{\sqrt{t^{1-\eta/2}}^3}|x|^{1-\alpha} ||u||_{C^{2,\alpha}_s(\R^n)}
\\&= c\sup_{x\in A(0, .5\sqrt{t}, 2\sqrt{t})} |x|^{2 - 2\eta}\frac{|x|^{-\tau+ \eta}}{t^{1-\eta/2}} ||u||_{C^{2,\alpha}_s(\R^n)} 
\\& \qquad \qquad +  c\sup_{x\in A(0, .5\sqrt{t}, 2\sqrt{t})}|x|^{3 - 3\eta/2}\frac{|x|^{-\tau + 3\eta/2}}{\sqrt{t^{1-\eta/2}}^3}||u||_{C^{2,\alpha}_s(\R^n)}
\\& \leq c \sqrt{t}^{-\tau + 3\eta/2}||u||_{C^{2,\alpha}_s(\R^n)},
\end{align*}
where are estimating as in the previous computations. 
\end{proof}
In particular, if $||u||_{C^{2,\alpha}_s(\R^n)} = 1$, Claim \ref{claim:tildeLclosetoL} implies that for all sufficiently large $t$ (large depending on $\tau, \eta, n, \alpha,$ and $s$) $|| L^{\hat g_t} - \tilde L||_{\op} < \mod(\tilde L)$, so that, by Lemma \ref{lemma:isomorphismthreshold}, $L^{\hat g_t}$ is an isomorphism. Now note that since $ \chi_t R(g_{t^{1-\eta/2}})$ and $R(\hat g_t)$ are both identically $0$ outside of a compact set, we have $\chi_t R(g_{t^{1-\eta/2}}) - R(\hat g_t)\in C^{0,\alpha}_{s-2}(\R^n)$. Therefore, since $L^{\hat g_t}$ is invertible, there exists some $v^t\in C^{2,\alpha}_s(\R^n)$ such that $L^{\hat g_t}v^t = \chi_t R(g_{t^{1-\eta/2}}) - R(\hat g_t)$. Moreover, $v^t$ is smooth by elliptic regularity. Let $u^t := v^t + 1$. Then $L^{\hat g_t} u^t = L^{\hat g_t}v^t + R(\hat g_t) = \chi_t R(g_{t^{1-\eta/2}})$.

It remains to show that $u^t$ is positive and decays to $1$ at infinity in such a way that the conformally changed metric $g_t'$ in the statement of Lemma \ref{lemma:conformaldeformation} is $C^2$-asymptotically flat, for large $t$. Arguing as above, by Remark \ref{rmk:gluedderivests}, Corollary \ref{cor:smoothdecay}, and the derivative estimates for $\chi_t$, if we set $s = -\tau$, then we have
\begin{align*}
& ||L^{\hat g_t} v^t||_{C^{0,\alpha}_s(\R^n)} = ||\chi_t R(g_{t^{1-\eta/2}}) - R(\hat g_t)||_{C^{0,\alpha}_s(\R^n)}
\\& =  \sup_{x\in A(0, .5\sqrt{t}, 2\sqrt{t})} \ell(x)^{-s}|(\chi_t R(g_{t^{1-\eta/2}}) - R(\hat g_t))(x)| 
\\& + \sup_{x\in A(0, .5\sqrt{t}, 2\sqrt{t})}\ell(x)^{\alpha - s}\sup_{y,z\in B(x, |x|/2), y\neq z} \frac{|(\chi_t R(g_{t^{1-\eta/2}}) - R(\hat g_t))(y) - (\chi_t R(g_{t^{1-\eta/2}}) - R(\hat g_t))(z)|}{d_{\R^n}(y,z)^{\alpha}}
\\& \leq c\sup_{x\in A(0, .5\sqrt{t}, 2\sqrt{t})} |x|^{-s}\frac{|x|^{-\tau}}{t^{1-\eta/2}} 
\\& \qquad \qquad + c\sup_{x\in A(0, .5\sqrt{t}, 2\sqrt{t})}|x|^{\alpha - s} (||\nabla(\chi_t R(g_{t^{1-\eta/2}}))||_{C^0(B(x, |x|/2))} + || \nabla R(\hat g_t) ||_{C^0(B(x, |x|/2))})|x|^{1 - \alpha}
\\& \leq c\sup_{x\in A(0, .5\sqrt{t}, 2\sqrt{t})} |x|^{-s}\frac{|x|^{-\tau}}{t^{1-\eta/2}} + c\sup_{x\in A(0, .5\sqrt{t}, 2\sqrt{t})}|x|^{\alpha - s} \frac{|x|^{-\tau}}{\sqrt{t^{1-\eta/2}}^3}|x|^{1 - \alpha}
\\& \leq c\sup_{x\in A(0, .5\sqrt{t}, 2\sqrt{t})} |x|^{-s}\frac{|x|^{-\tau}}{t^{1-\eta/2}} + c\sup_{x\in A(0, .5\sqrt{t}, 2\sqrt{t})}|x|^{- s} \frac{|x|^{-\tau + \eta/2}}{\sqrt{t^{1-\eta/2}}^3}|x|^{1 - \eta/2}
\\& \leq \frac{c}{t^{1-\eta/2}} + \frac{c}{t^{1-\eta}}.
\end{align*}
By (\ref{eq:minmoduniformlowerbound}) and Claim \ref{claim:tildeLclosetoL}, as $t\to \infty$ we have $\mod(L^{\hat g_t}) \geq \mod(\tilde L) - ||L^{\hat g_t} - \tilde L||_{\op} >0$ for some positive lower bound that does not depend on $t$, so
\begin{equation}\label{eq:vtconvergencerate}
\begin{split}
||v_t||_{C^{2,\alpha}_{-\tau}(\R^n)} & = ||(L^{\hat g_t})^{-1}(L^{\hat g_t} v^t)||_{C^{0,\alpha}} \leq ||(L^{\hat g_t})^{-1}||_{\op}||L^{\hat g_t} v^t||_{C^{0,\alpha}} 
\\&= \frac{1}{\mod(L^{\hat g_t})}||L^{\hat g_t} v^t||_{C^{0,\alpha}_{-\tau}(\R^n)} \leq c||L^{\hat g_t} v^t||_{C^{0,\alpha}_{-\tau}(\R^n)}
\\& \leq \frac{c}{t^{1-\eta}}\xrightarrow[t\to \infty]{} 0.
\end{split}
\end{equation}
Therefore,
\begin{equation}\label{eq:utconvto1}
u^t\xrightarrow[t\to \infty]{C^{2,\alpha}_{-\tau}} 1,
\end{equation}
so $u^t\to 1$ uniformly on $\R^n\setminus B(0,.5)$ as $t\to \infty$. Since $u^t$ is harmonic (with respect to the Euclidean Laplacian) on $B(0,1)$, by the maximum principle it also converges uniformly to $1$ on $\overline{B(0,.5)}$, so for sufficiently large $t$ (depending on $\eta, n$, and $|| \rho_0^\tau(g - \delta)||_{L^\infty(\R^n)}$, where $\rho_0$ is as in Lemma \ref{lemma:RDTFC0AF}) we have $u^t >0$ on $\R^n$. Also, since each $\hat g_t \equiv \delta$ outside of a compact set, the conformally changed metric $g_t' = (u^t)^{4/(n-2)}\hat g_t$ is classically asymptotically flat with decay rate $\tau$:
\begin{align*}
|(g_t')_{ij} - \delta_{ij}| + |x| |\nabla g_t'| + |x|^2|\nabla^2 g_t'| & \leq c|v^t (x)\delta_{ij}| + c |x||\nabla v^t (x)| + c|x|^2|\nabla^2 v^t(x)|
\\& = O(|x|^{-\tau}).
\end{align*}
\end{proof}

We now prove Theorem \ref{thm:masssequence}.
\begin{proof}[Proof of Theorem \ref{thm:masssequence}]
Take $T$ as in Lemma \ref{lemma:conformaldeformation}. For $r \geq \sqrt{T}$, let $g_{r^2}'$ be the metric given by the conclusion of Lemma \ref{lemma:conformaldeformation}. Fix some $r_0 > \sqrt{T}$, and note that by the assumptions on $T$, $\ell(x)$ agrees with $|x|$ outside of $B(0, r_0/2)$, where $\ell(s)$ is as in Theorem \ref{thm:LeeFredholm}. Let $r > 2r_0$ (eventually we take $r\to\infty$). By Lemma \ref{lemma:Bartnikmonotonicity} we have
\begin{equation}\label{eq:semilocalmasscomparison}
\begin{split}
M_{C^0}(g'_{r^2}, \varphi, r) - M_{C^0}(g_{r^{2-2\eta}}, \varphi, r) &= c(n)\int_{A(0, r_0, r )} R(g'_{r^2}) - R(g_{r^{2-\eta}}) + \nabla g'_{r^2} * \nabla g'_{r^2}  - \nabla g_{r^{2-\eta}}*\nabla g_{r^{2-\eta}} dx 
\\& + M_{C^0}(g'_{r^2}, \varphi, r_0) - M_{C^0}(g_{r^{2-\eta}}, \varphi, r_0).
\end{split}
\end{equation}
\begin{claim}
\begin{equation}
\int_{A(0, r_0, r )} R(g'_{r^2}) - R(g_{r^{2-\eta}}) \xrightarrow[r\to\infty]{} 0.
\end{equation}
\end{claim}
\begin{proof}[Proof of claim.]
This is because of Lemma \ref{lemma:finitemass} and the fact that
\begin{align*}
\int_{A(0, r_0, r )} |R(g'_{r^2}) - R(g_{r^{2-\eta}})| &= \int_{A(0, r_0, r )} (1 - \chi_{r^2} )R(g_{r^{2-\eta}})
\\& \leq \int_{A(0, r/2, 2r)}R(g_{r^{2-\eta}}) \xrightarrow[r\to\infty]{} 0.
\end{align*}
\end{proof}
\begin{claim}
\begin{equation*}
\int_{A(0, r_0, r)} \nabla g'_{r^2} * \nabla g'_{r^2}  - \nabla g_{r^{2-\eta}}*\nabla g_{r^{2-\eta}} \xrightarrow[r\to\infty]{} 0.
\end{equation*}
\end{claim}
\begin{proof}[Proof of claim.]
First observe that $\nabla u^{r^2} = \nabla v^{r^2}$ and (\ref{eq:vtconvergencerate}) implies that 
\begin{equation*}
|| v^{r^2}||_{C^{2,\alpha}_{-\tau}}(\R^n) \leq \frac{c}{r^{2-2\eta}}
\end{equation*}
so, for $|x| \geq r_0$,
\begin{equation*}
|\nabla u^{r^2}(x)|  \leq \frac{c\ell(x)^{-\tau - 1}}{r^{2-2\eta}} = \frac{c|x|^{-\tau - 1}}{r^{2-2\eta}}.
\end{equation*}
Combining this with Remark \ref{rmk:gluedderivests} we find
\begin{align*}
|\nabla g'_{r^2} | &= |\nabla [(u^t)^{4/(n-2)}\hat g_t]|
\\&= c(n)| \nabla u^{r^2} * (u^{r^2})^{4/(n-2) -1}*\hat g_{r^2} + (u^{r^2})^{4/(n-2)}\nabla \hat g_{r^2}|
\\& \leq \frac{c|x|^{-\tau - 1}}{r^{2 - 2\eta}} + \frac{c|x|^{-\tau}}{r^{2-\eta}}
\\& \leq \frac{c|x|^{-\tau}}{r^{2-2\eta}}.
\end{align*}

Therefore,
\begin{align*}
\int_{A(0, r_0, r)} | \nabla g'_{r^2} * \nabla g'_{r^2}  & - \nabla g_{r^{2-\eta}}*\nabla g_{r^{2-\eta}} | \leq \int_{A(0, r_0, r)} \frac{c|x|^{-2\tau}}{r^{2-2\eta}} 
\\& \leq c  \frac{\ell^{n - 1 -2\tau + 1}}{r^{2-2\eta}} \bigg|_{r_0}^{r} \xrightarrow[r\to \infty]{}0
\\& \leq c r^{n-2 - 2\tau + 2\eta} - c\frac{r_0^{n- 2\tau}}{r^{2-2\eta}} \xrightarrow[r\to \infty]{} 0,
\end{align*}
\end{proof}
by smallness of $\eta$.

\begin{claim}
\begin{equation*}
M_{C^0}(g'_{r^2}, \varphi, r_0) - M_{C^0}(g_{r^{2-\eta}}, \varphi, r_0) \xrightarrow[r\to \infty]{}0.
\end{equation*}
\end{claim}
\begin{proof}[Proof of claim.]
Since $\hat g_{r^2} \equiv g_{r^{2-\eta}}$ on $A(0, .9r_0, 1.1r_0)$ this follows from the definitions of $M_{C^0}(g'_{r^2}, \varphi, r_0)$ and $M_{C^0}(g_{r^{2-\eta}}, \varphi, r_0)$ and the fact that $u^{r^2}\xrightarrow[r\to \infty]{C^{0}_{\loc}(\R^n)} 1$ by Lemma \ref{lemma:conformaldeformation}.
\end{proof}
Combining the previous three claims with (\ref{eq:semilocalmasscomparison}) implies
\begin{equation}\label{eq:closenessatscaler}
M_{C^0}(g'_{r^2}, \varphi, r) - M_{C^0}(g_{r^{2-\eta}}, \varphi, r) \xrightarrow[r\to \infty]{} 0.
\end{equation}
\begin{claim}
\begin{equation}\label{eq:approximatingmodifiedmetric}
M_{C^0}(g'_{r^2}) - M_{C^0}(g'_{r^2}, \varphi, r) \leq \delta(r)\xrightarrow[r\to \infty]{} 0.
\end{equation}
\end{claim}
\begin{proof}[Proof of claim]
By Lemma \ref{lemma:Bartnikmonotonicity} we have that for some $s\in [.9, 1.1]$,
\begin{align*}
M_{C^0}(g'_{r^2}) - M_{C^0}(g'_{r^2}, \varphi, r) &= c(n) \int_{\R^n\setminus B(0, sr)} R(g'_{r^2}) + \nabla g'_{r^2}* \nabla g'_{r^2} 
\\& = c(n)\int_{\R^n\setminus B(0, sr)} \chi_{r^2} R(g_{r^{2-2\eta}}) + \nabla g'_{r^2}* \nabla g'_{r^2} 
\\& \leq \int_{A(0, sr, 2r)} R(g_{r^{2-\eta}}) + \int_{\R^n\setminus B(0, sr)} c|x|^{-2\tau + \eta - 2}dx, 
\end{align*}
where in the last step we are using the construction of $g'_{r^2}$ and (\ref{eq:vtconvergencerate}): on $\R^n\setminus B(0, 2r)$, $g'_{r^2} \equiv (u^{r^2})^{4/(n-2)}\delta$ so
\begin{equation*}
|\nabla g'_{r^2}(x)| \leq c_n|\nabla u^{r^2}(x)| \leq ||(u^{r^2} - 1)||_{C^{2,\alpha}_{-\tau}(\R^n)}|x|^{-\tau} \leq c|x|^{-\tau - 1}
\end{equation*}
by (\ref{eq:vtconvergencerate}), and on $A(0, sr, 2r)$ we have $.9r \leq |x| \leq 2r$ so
\begin{equation*}
\begin{split}
|\nabla g'_{r^2}(x)| &\leq c|\nabla u^{r^2}(x)| + |\nabla \chi_{r^2}(x)||g_{r^{2-\eta}}(x) - \delta| + |\nabla g_{r^{2-\eta}}(x)|
\\& \leq c|x|^{-\tau - 1} + \frac{c|x|^{-\tau}}{r} + \frac{c|x|^{-\tau}}{r^{1-\eta/2}} \leq c|x|^{-\tau -1 + \eta/2}.
\end{split}
\end{equation*}

A computation now shows that, since $\tau > (n-2)/2$ and by smallness of $\eta$,
\begin{equation*}
\int_{\R^n\setminus B(0, sr)} c|x|^{-2\tau +\eta - 2}dx \xrightarrow[r\to\infty]{} 0
\end{equation*}
and, since $M_{C^0}(g)$ is finite, Lemma \ref{lemma:finitemass} implies that
\begin{equation*}
\int_{A(0, sr, 10r)} R(g_{r^{2-\eta}}) \xrightarrow[r\to \infty]{} 0.
\end{equation*}
Thus we have
\begin{align*}
M_{C^0}(g'_{r^2}) - M_{C^0}(g'_{r^2}, \varphi, r) & \leq \int_{A(0, r, 2r)} R(g_{r^{2-\eta}}) + \int_{\R^n\setminus B(0, r)} c|x|^{-2\tau + \eta - 2}dx  \xrightarrow[r\to \infty]{} 0.
\end{align*}
\end{proof}
We now complete the proof of Theorem \ref{thm:masssequence}. First note that by Lemma \ref{lemma:massdistortionestimate}
\begin{equation*}
|M_{C^0}(g_{r^{2-\eta}}, \varphi, r) - M_{C^0}(g, \varphi_{r^{-\eta}}(0), r) | \leq cr^{n-2-2\tau} \xrightarrow[r\to \infty]{} 0
\end{equation*}
so
\begin{equation}\label{eq:rmasstoC0mass}
M_{C^0}(g) = \lim_{r\to\infty}M_{C^0}(g, \varphi_{r^{-\eta}}(0), r) = \lim_{r\to \infty}M_{C^0}(g_{r^{2-\eta}}, \varphi, r).
\end{equation} 

We now combine (\ref{eq:closenessatscaler}), (\ref{eq:approximatingmodifiedmetric}), and (\ref{eq:rmasstoC0mass}) to find that 
\begin{align*}
M_{C^0}(g'_{r^2}) - M_{C^0}(g) & = M_{C^0}(g'_{r^2}) -  M_{C^0}(g'_{r^2}, \varphi, r)
\\& + M_{C^0}(g'_{r^2}, \varphi, r) - M_{C^0}(g_{r^{2-\eta}}, \varphi, r)
\\& +  M_{C^0}(g_{r^{2-\eta}}, \varphi, r) - M_{C^0}(g)
\\& \leq \delta(r) \xrightarrow[r\to\infty]{} 0
\end{align*}
and hence, applying Theorem \ref{thm:C2PMT} to each $g'_{r^2}$, and using the fact that $M_{C^0}(g'_{r^{2}}) = m_{ADM}(g'_{r^2})$, we have
\begin{equation*}
M_{C^0}(g) \geq M_{C^0}(g'_{r^2}) - \delta(r) \geq -\delta(r)\xrightarrow[r\to\infty]{} 0.
\end{equation*}
\end{proof}

\appendix
\section{$C^0$-asymptotic flatness is preserved along Ricci-DeTurck flow}\label{appendix:RDTFweightedXnorm}

The purpose of this section is to prove Lemma \ref{lemma:RDTFC0AF} and Lemma \ref{lemma:RDTFL2distortion}. We refer the reader to \cite{McFeronSzekelyhidi12} and \cite{ChuLeeWan26} for short-time decay estimates for the Ricci-DeTurck flow. Throughout this section we fix some $0 < T < 1$. Let $\psi$ be some $1$-Lipschitz function and let $\lambda = 1/\sqrt{T}$. 

\begin{definition}\label{def:weightednorms}
Given, $T$ and $\psi$, define the weighted norms $||\cdot||_{\tilde X_T}, ||\cdot||_{\tilde Y^0_T}$, and $||\cdot||_{\tilde Y^1_T}$ by
\begin{equation*}
\begin{split}
|| h||_{\tilde X_T} & := \sup_{0 < t < T}|| \exp(-\lambda\psi) h||_{L^\infty(\R^n)} 
\\& + \sup_{x \in \R^n}\sup_{0 < r < \sqrt{T}} \exp(-\lambda \psi(x))\left( r^{-n/2}|| \nabla h||_{L^2(B(x,r)\times (0,r^2))} + r^{\tfrac{2}{n+4}}||\nabla h||_{L^{n+4}(B(x,r)\times (r^2/2, r^2))} \right)\\
|| h||_{\tilde Y_T^0} &:= \sup_{x\in \R^n}\sup_{0 < r < \sqrt{T}} \exp(-\lambda \psi(x))\left( r^{-n}||h||_{L^1(B(x,r)\times (0,r^2))} + r^{\tfrac{4}{n+4}}||h||_{L^{\tfrac{n+4}{2}}(B(x,r)\times (r^2/2, r^2))} \right)\\
|| h||_{\tilde Y_T^1} &:= \sup_{x\in \R^n}\sup_{0 < r < \sqrt{T}}\exp(-\lambda \psi(x))\left( r^{-n/2}||  h||_{L^2(B(x,r)\times (0,r^2))} + r^{\tfrac{2}{n+4}}|| h||_{L^{n+4}(B(x,r)\times (r^2/2, r^2))} \right)\\
||h||_{\tilde Y_T} &= \inf\{ ||f_0||_{\tilde Y^0_T} + ||f_1||_{\tilde Y^1_T} : f = f_0 + \nabla^* f_1\}
\end{split}
\end{equation*}
\end{definition}

\begin{theorem}\label{thm:weighteddifference}
There exists $\varepsilon = \varepsilon(n) > 0$, less than or equal to the threshold $\varepsilon$ in Lemma \ref{lemma:KL}, such that the following is true: Suppose $g$ is a continuous Riemannian metric on $\R^n$ such that $|| g- \delta||_{L^\infty(\R^n)} < \varepsilon$, and suppose that $g_t$ is a solution to the Ricci-DeTurck flow starting from $g$ in the sense of Lemma \ref{lemma:KL}. Let $\chi: \R^n\to [0,1]$ be a smooth cutoff function identically equal to 1 on $B(0,4)$ and identically equal to 0 outside of $B(0,5)$. Let $g_t'$ be Ricci-DeTurck flow starting from $g':= \chi g + (1-\chi)\delta$ in the sense of  Lemma \ref{lemma:KL}. Then
\begin{equation*}
|| g_t - g_t'||_{\tilde X_T} \leq c(n)|| \exp(-\lambda \psi) (g - g')||_{L^\infty(\R^n)}.
\end{equation*}
\end{theorem}

To prove Theorem \ref{thm:weighteddifference}, we establish three lemmata. These are weighted versions of \cite[Lemmata 2.2, 4.1, and 4.2]{KochLamm12}. First note that, since $\psi$ is 1-Lipschitz, if $y\in B(x, 2r)$ and $r < \sqrt{T}$ then
\begin{equation}\label{eq:weightcomparison}
\exp(-\lambda\psi(x)) \leq \exp(-\lambda \psi(y)) \exp(2\lambda r) \leq e^2 \exp(-\lambda\psi(y)).
\end{equation}
\begin{lemma}\label{lemma:homogeneous}
Suppose $u_0$ is some bounded, continuous $(0,2)$-tensor field on $\R^n$, and for $t>0$ let $u_t$ be the smooth $(0,2)$-tensor field given by
\begin{equation*}
u(x,t) = \int_{\R^n} \Phi(x,t;y,0)u_0(y)dy.
\end{equation*}
Then
\begin{equation*}
|| u_t||_{\tilde X_T} \leq c(n)||e^{-\lambda \psi}u_0||_{L^\infty(\R^n)}.
\end{equation*}
\end{lemma}
\begin{proof}
We first prove the pointwise bounds
\begin{equation}\label{eq:ptwiseC0bound}
|\exp(-\lambda \psi(x))u(x,t)| \leq c(n)|| \exp(-\lambda \psi) u_0||_{L^\infty(\R^n)}
\end{equation}
and
\begin{equation}\label{eq:ptwiseC1bound}
|\exp(-\lambda \psi(x))\nabla u(x,t)| \leq \frac{c(n)}{\sqrt{t}}|| \exp(-\lambda \psi) u_0||_{L^\infty(\R^n)}
\end{equation}
for all $x\in \R^n$ and $0 < t < T$. Because $\psi$ is 1-Lipschitz, we have
\begin{align*}
|\exp(-\lambda \psi(x))u(x,t)| &= \int_{\R^n} \Phi(x,t;y,0)\exp(-\lambda \psi(y))\exp(\lambda \psi(y)-\lambda \psi(x))|u_0(y)|dy
\\& \leq \int_{\R^n} (4\pi t)^{-n/2}\exp\left(-\frac{|x-y|^2}{4t} + \frac{|x-y|}{\sqrt{T}}\right)dy ||\exp(-\lambda \psi)u_0||_{L^\infty(\R^n)}
\\& \leq c(n) ||\exp(-\lambda \psi)u_0||_{L^\infty(\R^n)},
\end{align*}
where in the last step we are using the fact that $t \leq T$. This proves (\ref{eq:ptwiseC0bound}), and the proof of (\ref{eq:ptwiseC1bound}) is similar:
\begin{align*}
|\exp(-\lambda \psi(x))\nabla u(x,t)| & \leq \int_{\R^n} |\nabla \Phi(x,t;y,0)|\exp(-\lambda \psi(y))\exp(\lambda \psi(y)-\lambda \psi(x))|u_0(y)|dy
\\& \leq \int_{\R^n} \frac{|x-y|}{2t} (4\pi t)^{-n/2}\exp\left(-\frac{|x-y|^2}{4t} + \frac{|x-y|}{\sqrt{T}}\right)dy ||\exp(-\lambda \psi)u_0||_{L^\infty(\R^n)}
\\& \leq \frac{c(n)}{\sqrt{t}}||\exp(-\lambda \psi)u_0||_{L^\infty(\R^n)}.
\end{align*}
We now estimate $||u_t||_{\tilde X_T}$. The $L^\infty$ term is handled by (\ref{eq:ptwiseC0bound}). To estimate the remaining terms, we argue as in \cite[Lemma 2.2]{KochLamm12}: Fix $x\in \R^n$ and $0 < r < \sqrt{T}$. pair $u_t$ with a cutoff function $\eta$ that is identically $1$ on $B(x, r)$, identically $0$ outside of $B(x, 2r)$, with $|\nabla \eta| \leq 10/r$ and integrate by parts over $\R^n \times (0,r^2)$ and apply Young's inequality to find that
\begin{equation*}
\exp(-2\lambda \psi(x)) \int_0^{r^2}\int_{B(x,r)} |\nabla u|^2 \leq 8\exp(-2\lambda \psi(x))\int_{B(x, 2r)} |h_0|^2 + \frac{8\exp(-2\lambda \psi(x))}{r^2}\int_0^{r^2}\int_{B(x, 2r)} |h|^2.
\end{equation*}
Applying (\ref{eq:weightcomparison}) we find
\begin{align*}
\exp(-2\lambda \psi(x)) \int_0^{r^2}\int_{B(x,r)} |\nabla u|^2 &\leq c\int_{B(x, 2r)} |\exp(-\lambda \psi)h_0|^2 + \frac{c}{r^2}\int_0^{r^2}\int_{B(x, 2r)} |\exp(-\lambda \psi)h|^2
\\& \leq c(n)r^n||\exp(-\lambda\psi) h_0||_{L^\infty(\R^n)}^2 + \frac{c(n)}{r^2}r^2r^n||\exp(-\lambda\psi) h||_{L^\infty(B(x,2r))}^2
\\& \leq c(n)r^n||\exp(-\lambda\psi) h_0||_{L^\infty(\R^n)}^2, 
\end{align*}
where in the last step we have used (\ref{eq:ptwiseC0bound}).

The $L^{n+4}$-term follows from (\ref{eq:ptwiseC1bound}) by integration and (\ref{eq:weightcomparison}):
 \begin{align*}
\exp(-\lambda\psi(x))r^{\frac{2}{n+4}}\left(\int_{r^2/2}^{r^2}\int_{B(x,r)} |\nabla u|^{n+4}\right)^{\frac{1}{n+4}} & \leq cr^{\frac{2}{n+4}}\left(\int_{r^2/2}^{r^2}\int_{B(x,r)}\exp(-(n+4)\lambda\psi(y))|\nabla u|^{n+4}\right)^{\frac{1}{n+4}}
 \\& \leq cr^{\frac{2}{n+4}}\left(\int_{r^2/2}^{r^2}\int_{B(x,r)} \frac{c(n)}{\sqrt{s}^{n+4}}||\exp(-\lambda\psi) u_0||_{L^\infty(\R^n)}^{n+4}\right)^{\frac{1}{n+4}}
 \\& \leq c(n)r^{\frac{2}{n+4}}||\exp(-\lambda\psi) u_0||_{L^\infty(\R^n)}(r^{n+2}r^{-n-4})^{\frac{1}{n+4}} 
 \\& = c(n)||\exp(-\lambda\psi) u_0||_{L^\infty(\R^n)}.
 \end{align*}
\end{proof}

\begin{lemma}\label{lemma:hard}
Suppose $Q(x,t) \in \tilde Y_T$ and that $w_t$ is the family of $(0,2)$-tensors on $\R^n$ given by
\begin{equation*}
w_t(x) = \int_0^t \int_{\R^n} \Phi(x,t;y,s)Q(y,s).
\end{equation*}
Then
\begin{equation*}
|| w||_{\tilde X_T} \leq c(n) || Q||_{\tilde Y_T}.
\end{equation*}
\end{lemma}
\begin{proof}
The proof is similar to that of \cite[Lemma 4.2]{KochLamm12}. Write $Q = Q^0 + \nabla Q^1$. We first estimate the $L^\infty$ part of the $||w||_{\tilde X_T}$. Let $\Omega(y,t):= B(y,t) \times [\tfrac{t}{2}, t]$. We have: 
\begin{align*}
\sup_{0<t<r^2}||\exp(-\lambda\psi) w||_{L^\infty(B(x,r))} & \leq \sup_{\substack{0<t<r^2 \\ y\in B(x,r)}}\left|\exp(-\lambda\psi(y))\int_{\Omega(y,t)}\Phi(y,t;z,s)Q(z,s)dzds\right| 
\\& \qquad \qquad + \left|\exp(-\lambda\psi(y))\int_{M\times[0,t]\setminus\Omega(y,t)}\Phi(y,t;z,s)Q(z,s)dzds\right|
\\& =: \sup_{\substack{0<t<r^2 \\ y\in B(x,r)}} I + \sup_{\substack{0<t<r^2 \\ y\in B(x,r)}}II.
\end{align*}
For any $0<t <r^2$ and $y\in B(x,r)$ we proceed as follows: Using H\"older's inequality, we find that
\begin{align*}
I &\leq \omega(y)||\Phi(y,t;\cdot,\cdot)||_{L^{\tfrac{n+4}{n+2}}(\Omega(y,t))}||Q^0||_{L^{\tfrac{n+4}{2}}(\Omega(y,t))} 
\\& \qquad\qquad + \omega(y)||\nabla \Phi(y, t; \cdot,\cdot)||_{L^{\tfrac{n+4}{n+3}}(\Omega(y,t))}||Q^1||_{L^{n+4}(\Omega(y,t))}
\\& \qquad \qquad \qquad \leq c||Q||_{\tilde Y(y,\sqrt{t})},
\end{align*}
where the last inequality is due to \cite[Lemma 2.1]{KochLamm12} as follows:
\begin{align*}
\left(\int_{t/2}^{t}\int_{B(y,\sqrt{t})}|\Phi(y,t; z,s)|^{\frac{n+4}{n+2}}dzds\right)^{\tfrac{n+2}{n+4}} & \leq c\left(\int_0^{t/2}\int_{B(y,\sqrt{t})}(|y-z| + \sqrt{s})^{-n\frac{n+4}{n+2}}dzds\right)^{\tfrac{n+2}{n+4}}
\\& \leq c\left(t^{\frac{n}{2}}t^{-\frac{n(n+4)}{2(n+2)} + 1}\right)^{\tfrac{n+2}{n+4}} = c\sqrt{t}^{\frac{4}{n+4}},
\end{align*}
and similarly for $\nabla \Phi$.

We now estimate $II$. Let $\{z_i\}$ be a maximal collection of points in $\R^n$ such that the balls $B(z_i, \sqrt{t}/2)$ are pairwise disjoint. Observe that $\{B(z_i, \sqrt{t})\}_{i=1}^{\infty}$ is a cover of $\R^n$. By \cite[Lemma 2.1]{KochLamm12} (after parabolically rescaling) and integrating by parts we have
\begin{align*}
II & \leq \left|\exp(-\lambda\psi(x))\int_{\R^n\times [0,t]\setminus \Omega(x,t)}\Phi(x,t;z,s)Q(z,s)dzds\right|
\\& \leq \left|\sum_{i=0}^{\infty}\exp(-\lambda\psi(x))\int_0^t\int_{B(z_i,\sqrt{t})}\exp\left(-\frac{|x-z|^2}{ct}\right)\sqrt{t}^{-n}\left(Q^0(z,s) + t^{-1/2}Q^1(z,s)\right)dzds\right|
\\& \leq \left|\sum_{i=0}^{\infty}\exp\left(-\frac{|x - z_i|^2}{ct} + \frac{|x - z_i|}{\sqrt{T}}\right)\exp(-\lambda\psi(z_i))\int_0^t\int_{B(z_i,\sqrt{t})}\sqrt{t}^{-n}(Q^0(z,s) + t^{-1/2}Q^1(z,s))dzds\right|
\\& \leq c\left|\sum_{i=0}^{\infty}\exp\left(-\frac{|x - z_i|^2}{2ct}\right)\exp(-\lambda\psi(z_i))\int_0^t\int_{B(z_i,\sqrt{t})}\sqrt{t}^{-n}(Q^0(z,s) + t^{-1/2}Q^1(z,s))dzds\right|
\\& \leq c\sum_{i=0}^{\infty}\exp\left(-\frac{|x - z_i|^2}{2ct}\right)\exp(-\lambda\psi(z_i))(\sqrt{t}^{-n}||Q^0||_{L^1(B(z_i, \sqrt{t})\times (0,t))} + \sqrt{t}^{-n/2}||Q^1||_{L^2(B(z_i, \sqrt{t})\times (0,t))})
\\& \leq c||Q||_{\tilde Y_T}.
\end{align*}
This establishes the $L^\infty$ estimate. To estimate the $L^2$ term, fix $x\in \R^n$ and $0 < r < \sqrt{T}$, and multiply by a cutoff function and integrate as in the proof of Lemma \ref{lemma:homogeneous} to find that
\begin{equation*}
\int_0^{r^2}\int_{B(x,r)}|\nabla w|^2 \leq \frac{c}{r^2}\int_0^{r^2}\int_{B(x,2r)}|w|^2 + c\int_{0}^{r^2}\int_{B(x,2r)} |Q^0||w| + c\int_0^{r^2}\int_{B(x,2r)} |Q^1|^2.
\end{equation*}
Then, by (\ref{eq:weightcomparison}) and applying the $L^\infty$ estimate to $w$, we have
\begin{align*}
&\exp(-2\lambda\psi(x)) r^{-n}\int_0^{r^2}\int_{B(x,r)}|\nabla w|^2   \leq c(n)|| \exp(-\lambda\psi) w||_{L^\infty(B(x, 2r) \times (0, r^2))}^2 
\\& \qquad +  c(n)||\exp(-\lambda\psi)w||_{L^\infty(B(x, 2r)\times (0, r^2))}\exp(-\lambda\psi(x))r^{-n}||Q^0||_{L^1(B(x, 2r)\times(0, r^2))}
\\& \qquad \qquad + c\exp(-2\lambda\psi(x))r^{-n}||Q^1||_{L^2(B(x, 2r)\times (0, r^2))}
\\& \leq c(n)|| \exp(-\lambda\psi) w||_{L^\infty(B(x, 2r) \times (0, r^2))}
\\& \qquad +  c(n)||\exp(-\lambda\psi)w||_{L^\infty(B(x, 2r)\times (0, r^2))}\sum_{i=1}^k\exp(-\lambda\psi(z_i))\exp(2r/\sqrt{T})r^{-n}||Q^0||_{L^1(B(z_i, r)\times(0, r^2))}
\\& \qquad \qquad + c\sum_{i = 1}^{k}\exp(-2\lambda\psi(z_i))\exp(2r/\sqrt{T})r^{-n}||Q^1||_{L^2(B(z_i, r)\times (0, r^2))}
\\& \leq c(n)||Q||_{\tilde Y_T},
\end{align*}
where in the penultimate step we are taking $\{B(z_i, r)\}_{i=1}^{k}$ to be a covering of $B(x, 2r)$ by $k(n)$-many balls, and using the fact that $r \leq \sqrt{T}$. This establishes the $L^2$ estimate.

To estimate the $L^{n+4}$-term, we argue as in the estimate for $II$ (noting that if $(y,t)\in B(x, r)\times (r^2/2, r^2)$ then $B(y, \sqrt{t})\times (t/2 ,t) \subset B(x, 2r)\times (r^2/4, r^2)$ to see that
\begin{align*}
&\sup_{(y,t)\in B(x, r)\times (r^2/2, r^2)}\exp(-\lambda\psi(y))\bigg| \int_{\R^n\times (0,t)\setminus B(x, 2r)\times(r^2/4, r^2)} \nabla \Phi(y,t;z,s) Q^0(z,s) 
\\& \qquad \qquad + \nabla^2 \Phi(y,t;z,s)Q^1(z,s)dzds\bigg|
\\& \leq c(n)\sqrt{t}^{-1}||Q||_{\tilde Y_T}
\end{align*}
and hence, by (\ref{eq:weightcomparison})
\begin{align*}
&\exp(-\lambda \psi(x))r^{\tfrac{2}{n+4}}\bigg|\bigg| \int_{\R^n\times (0,t)\setminus B(x, 2r)\times(r^2/4, r^2)} \nabla \Phi(\cdot,\cdot ;z,s) Q^0(z,s) 
\\& \qquad \qquad + \nabla^2 \Phi(\cdot, \cdot;z,s)Q^1(z,s) \bigg|\bigg|_{L^{n+4}(B(x,r)\times (r^2/2, r^2))}
\\& \leq c r^{\tfrac{2}{n+4}}\bigg|\bigg| \exp(-\lambda \psi(\cdot))\int_{\R^n\times (0,t)\setminus B(x, 2r)\times(r^2/4, r^2)} \nabla \Phi(\cdot,\cdot ;z,s) Q^0(z,s) 
\\& \qquad \qquad + \nabla^2 \Phi(\cdot, \cdot;z,s)Q^1(z,s) \bigg|\bigg|_{L^{n+4}(B(x,r)\times (r^2/2, r^2))}
\\& \leq c(n)r^{\tfrac{2}{n+4}}||Q||_{\tilde Y_T}||_{L^{n+4}(B(x,r)\times (r^2/2, r^2))}
\\& \leq c(n)||Q||_{\tilde Y_T}.
\end{align*}

Therefore, one may assume that $Q^0$ and $Q^1$ have support contained in $B(x, 2r) \times (r^2/4, r^2)$. The rest of the $L^{n+4}$ estimate then follows as in \cite[Lemma 4.2]{KochLamm12}, since the weight does not need to be passed inside the integral in this region.
\end{proof}

\begin{lemma}\label{lemma:easy}
Let $g_t$ and $g_t'$ be Ricci-DeTurck flows as in the statement of Theorem \ref{thm:weighteddifference}. Let $h = g_t - \delta$ and $h' = g_t'-\delta$. Then
\begin{equation*}
|| Q^0[h] - Q^0[h'] + \nabla^* Q^1[h] - \nabla^* Q^1[h']||_{\tilde Y_T} \leq c(n) (|| h||_{X_T} + ||h'||_{X_T}) || h - h'||_{\tilde X_T}.
\end{equation*}
\end{lemma}
\begin{proof}
This is immediate from Definition \ref{def:weightednorms} and (\ref{eq:weightcomparison}), using the fact that
\begin{equation*}
\begin{split}
|Q^0[h] - Q^0[h']|  &= |\nabla h * \nabla h  - \nabla h * \nabla h' + \nabla h * \nabla h' - \nabla h' * \nabla h' | \leq c(n) (|h| + |h'|)|| h-h'|\\
|Q^1[h] - Q^1[h']| &= |h * \nabla h - h * \nabla h' + h * \nabla h' - h'*\nabla h'| \leq c(n)(|h| |\nabla (h - h')| + |h - h'||\nabla h'|).
\end{split}
\end{equation*}
For example, by (\ref{eq:weightcomparison}), for $r< \sqrt{T}$ and $x\in \R^n$ we have
\begin{align*}
\exp(-\lambda\psi(x))&r^{-n/2}||Q^1[h] - Q^1[h'] ||_{L^2(B(x,r)\times (0,r^2))} 
\\& \leq c(n)||h||_{L^\infty(\R^n\times (0,T))}\exp(-\lambda\psi(x))r^{-n/2}||\nabla(h - h')||_{L^2(B(x,r)\times (0, r^2))}
\\& \qquad \qquad + c(n)||\exp(-\lambda\psi)(h - h')||_{L^\infty(\R^n\times (0,T))}r^{-n/2}||\nabla h'||_{L^2(B(x,r)\times (0, r^2))}
\\& \leq c(n)(||h||_{X_T} + ||h'||_{X_T})|| h - h'||_{\tilde X_T}. 
\end{align*}
\end{proof}

We are now ready to prove Theorem \ref{thm:weighteddifference}.
\begin{proof}[Proof of Theorem \ref{thm:weighteddifference}]
As in Lemma \ref{lemma:easy}, let $h_t = g_t - \delta$ and $h_t' = g_t' - \delta$. We have $\partial(h_t - h_t') = \Delta (h_t - h_t') + Q[h_t] - Q[h_t']$ so that
\begin{equation*}
h_t(x) - h_t'(x) = \int_{\R^n}\Phi(x,t;y,0) (h_0(y)- h_0'(y))dy + \int_0^t\int_{\R^n}\Phi(x,t;y,s)(Q[h_s](y) - Q[h_s'](y))dyds
\end{equation*}
We apply Lemma \ref{lemma:homogeneous} to the first term, and apply Lemma \ref{lemma:hard} and then Lemma \ref{lemma:easy} to the second term to find that 
\begin{align*}
|| h_t - h_t'||_{\tilde X_T} &\leq \left|\left|\int_{\R^n}\Phi(\cdot,t;y,0) (h_0(y)- h_0'(y))dy\right|\right|_{\tilde X_T}
\\& + \left|\left| \int_0^t\int_{\R^n}\Phi(\cdot,t;y,s)(Q[h_s](y) - Q[h_s'](y))dyds\right|\right|_{\tilde X_T}
\\& \leq c(n)|| \exp(-\lambda \psi) (h_0 ' h_0')||_{L^\infty(\R^n)} + c(n)|| Q[h_t] - Q[h_t']||_{\tilde Y_T}
\\& \leq c(n)|| \exp(-\lambda \psi) (h_0 - h_0')||_{L^\infty(\R^n)} + c(n)(|| h_t||_{X_T} + ||h_t'||_{X_T})|| h_t - h_t'||_{\tilde X_T}.
\end{align*}
In particular, if $\varepsilon$ is sufficiently small depending on $n$, then Lemma \ref{lemma:KL} implies that $|| h_t||_{X_T} + ||h_t'||_{X_T} \leq c(n)\varepsilon$ so by subtraction we have
\begin{equation*}
|| h_t - h_t'||_{\tilde X_T} \leq c(n)|| \exp(-\lambda \psi) (h_0 - h_0')||_{L^\infty(\R^n)}, 
\end{equation*} 
with $c(n)$ adjusted. This proves the result.
\end{proof}

We now prove Lemma \ref{lemma:RDTFL2distortion}.
\begin{proof}[Proof of Lemma \ref{lemma:RDTFL2distortion}]
Let $\varepsilon'''$ denote the constant $\varepsilon$ from the statement of Theorem \ref{thm:weighteddifference}. Let $\bar r = 1$. We will increase $\bar r$ as needed throughout the proof. Now fix $0 < T < 1/2$. We first prove an estimate for $B(0,1)\times(0, T)$. Let $\psi(x) = \min\{d(x, B(0,1)), 3\}$, and let $g_t'$ be as in the statement of Theorem \ref{thm:weighteddifference}. Let $h_t = g_t - \delta$ and $h_t' = g_t' - \delta$.
\begin{align*}
|| \nabla h_t||_{L^2(B(0, 1)\times (0, T))} & \leq || \nabla h_t'||_{L^2(B(0, 1)\times (0, T))} +|| \nabla (h_t - h_t')||_{L^2(B(0, 1)\times (0, T))} 
\\& \leq || \nabla h_t'||_{L^2(B(0, 1)\times (0, 1))} + \sum_{i=1}^{k(n, T)}|| \nabla(h_t - h_t')||_{L^2(B(z_i, \sqrt{T})\times (0, T))}
\\& \leq || \nabla h_t'||_{L^2(B(0, 1)\times (0, 1))} + \sqrt{T}^{n/2}\sum_{i=1}^{k(n, T)}\sqrt{T}^{-n/2}|| \nabla(h_t - h_t')||_{L^2(B(z_i, \sqrt{T})\times (0, T))}
\\& \leq ||h_t'||_{X_1} + \sqrt{T}^{n/2}\sum_{i=1}^{k(n, T)}\exp(-\lambda \psi(z_i))\sqrt{T}^{-n/2}|| \nabla(h_t - h_t')||_{L^2(B(z_i, \sqrt{T})\times (0, T))},
\end{align*}
where we choose $\{B(z_i, \sqrt{T}/2)\}_{i=1}^{k(n,T)}$ to be a maximal disjoint collection of balls with $z_i\in B(0,1)$ so that $\{B(z_i, \sqrt{T})\}_{i=1}^{k(n,T)}$ is a cover of $B(0,1)$, and in the last step we are using the fact that if $z_i \in B(0,1)$ then $\psi(z_i) = d(z_i, B(0,1)) = 0$.
By Lemma \ref{lemma:KL} the first summand is bounded by
\begin{equation*}
||h_t'||_{X_1} \leq c(n)||h_0'||_{L^\infty(\R^n)} \leq c(n)||h_0||_{L^\infty(B(0,4))}.
\end{equation*}
The second summand is bounded by Theorem \ref{thm:weighteddifference} (with $T$ replaced by $2T$):
\begin{equation*}
\begin{split}
\sqrt{T}^{n/2}\sum_{i=1}^{k(n, T)}&\exp(-\lambda \psi(z_i))\sqrt{T}^{-n/2}|| \nabla(h_t - h_t')||_{L^2(B(z_i, \sqrt{T})\times (0, T))}  \leq k(n,T)\sqrt{T}^{n/2} || h_t - h_t''||_{\tilde X_{2T}}
\\& \leq c(n)k(n,T)\sqrt{T}^{n/2}|| \exp(-\psi/\sqrt{2T})(h_0 - h_0')||_{L^\infty(\R^n)}
\\& \leq c(n)k(n, T)\sqrt{T}^{n/2}\exp(-3/\sqrt{2T})||h_0||_{L^\infty(\R^n)}
\\& \leq c(n)\sqrt{T}^{-n/2}\exp(-3/\sqrt{2T})||h_0||_{L^\infty(\R^n)},
\end{split}
\end{equation*}
where we have used the fact that $\psi \geq 3$ whenever $h_0 \neq h_0'$ and that $||h_0'||_{L^\infty(\R^n)} \leq ||h_0||_{L^\infty(\R^n)}$, and in the last step we are estimating
\begin{equation*}
k(n,T) \leq c(n)\sqrt{T}^{-n}
\end{equation*}
since, because the $B(z_i, \sqrt{T}/2)$ are all disjoint, we have
\begin{equation*}
c(n)k(n,T)(\sqrt{T}/2)^n \leq \sum_{i=1}^{k(n,T)}|B(z_i, \sqrt{T}/2)| \leq |B(0, 2)|.
\end{equation*}

Combining the estimates for these two terms with the above calculation, we find
\begin{equation}\label{eq:scale1L2est}
|| \nabla h_t||_{L^2(B(0, 1)\times (0, T))} \leq c(n)||h_0||_{L^\infty(B(0,4))} + c(n)\sqrt{T}^{-n/2}\exp(-3/\sqrt{2T})||h_0||_{L^\infty(\R^n)}.
\end{equation}

Now let $r > \bar r$. We now apply (\ref{eq:scale1L2est}) to $\hat g_t(y) := g_{.01 r^2t}(.1ry + x)$, which by Remark \ref{rmk:parabolicrescaling} is a Ricci-DeTurck flow in the sense of Lemma \ref{lemma:KL} starting from $\hat g_0(y):= g_0(.1ry + x)$ to find that
\begin{align*}
&r^{-n/2}||\nabla g_t||_{L^2(B(x, .1r)\times (0, r^{2-\eta}))} = || \nabla \hat g_t||_{L^2(B(0,1)\times (0, 100r^{-\eta}))} 
\\& \leq c(n)|| \hat g_0 - \delta ||_{L^\infty(B(0, 4))} + c(n)\sqrt{100r^{-\eta}}^{-n/2}\exp(-3/\sqrt{200r^{-\eta}})||\hat h_0 - \delta||_{L^\infty(\R^n)}
\\& \leq c(n) ||g_0 - \delta||_{L^\infty(B(x, .4r))} + c(n)r^{n\eta/2}\exp(-\tfrac{3}{\sqrt{200}} r^{\eta/2})
\\&\leq c(n)r^{-\tau},
\end{align*}
where in the last step we are using that $r \geq \bar r(n, \eta, \tau)$, and choosing $\bar r$ sufficiently large so that for all $r \geq \bar r$,
\begin{equation*}
r^{n\eta/2}\exp(-\tfrac{3}{\sqrt{200}} r^{\eta/2}) \leq 2r^{-\tau}.
\end{equation*}
\end{proof}

\bibliographystyle{plain}
\bibliography{C0PMTbib}

\begin{thebibliography}{10}

\bibitem{AntonelliFogagnoloNardulliPozzetta26}
Gioacchino Antonelli, Mattia Fogagnolo, Stefano Nardulli, and Marco Pozzetta.
\newblock Positive mass and isoperimetry for continuous metrics with
  nonnegative scalar curvature.
\newblock {\em Annales de l'Institut Henri Poincar\'e C}, to appear.

\bibitem{Appleton18}
Alexander Appleton.
\newblock Scalar curvature rigidity and {R}icci {D}e{T}urck flow on
  perturbations of {E}uclidean space.
\newblock {\em Calculus of Variations and Partial Differential Equations}, 57,
  2018.

\bibitem{ArnowittDeserMisner61}
R.~Arnowitt, S.~Deser, and C.~W. Misner.
\newblock Coordinate invariance and energy expressions in general relativity.
\newblock {\em Phys. Rev.}, 122(3):997 -- 1006, May 1961.

\bibitem{BamlerKleiner22}
Richard Bamler and Bruce Kleiner.
\newblock Uniqueness and stability of {R}icci flow through singularities.
\newblock {\em Acta Mathematica}, 228, 2022.

\bibitem{Bamler14}
Richard~H. Bamler.
\newblock Stability of hyperbolic manifolds with cusps under {R}icci flow.
\newblock {\em Advances in Mathematics}, 263:412 -- 467, 2014.

\bibitem{Bamler16}
Richard~H. Bamler.
\newblock A {R}icci flow proof of a result by {G}romov on lower bounds for
  scalar curvature.
\newblock {\em Mathematical Research Letters}, 23(2):325 -- 337, 2016.

\bibitem{Bartnik86}
R.~Bartnik.
\newblock The mass of an asymptotically flat manifold.
\newblock {\em Communications on Pure and Applied Mathematics}, 39:661--693,
  1986.

\bibitem{BenattiFogagnolo26}
Luca Benatti and Mattia Fogagnolo.
\newblock An isoperimetric characterization of a new {A}{D}{M}-like mass for
  ${C}^0$-asymptotically flat manifolds.
\newblock arXiv:2608.27103 [math.DG]. \url{https://arxiv.org/abs/2608.27103}.
\newblock Accessed September 21, 2026.

\bibitem{BiHaoHeShiZhu26}
Yuchen Bi, Tianze Hao, Shihang He, Yuguang Shi, and Jintian Zhu.
\newblock A proof for the {R}iemannian positive mass theorem up to dimension
  19.
\newblock arXiv:2603.02769 [math.DG]. \url{https://arxiv.org/abs/2603.02769}.
\newblock Accessed June 23, 2026.

\bibitem{Bray11}
Hubert~L. Bray.
\newblock On the {P}ositive {M}ass, {P}enrose, and {Z}{A}{S} inequalities in
  {G}eneral {D}imension.
\newblock In Hubert~L. Bray and William P.~Minicozzi II, editors, {\em Surveys
  in Geometric Analysis and Relativity}, volume~20 of {\em Advanced Lectures in
  Mathematics}, pages 1 -- 28. International Press of Boston, Inc., 2011.

\bibitem{PBG26}
Paula Burkhardt-Guim.
\newblock Smoothing ${L}^\infty$ {R}iemannian metrics with nonnegative scalar
  curvature outside of a singular set.
\newblock {\em Mathematische Annalen}, 395(57).

\bibitem{PBG19}
Paula Burkhardt-Guim.
\newblock Pointwise lower scalar curvature bounds for ${C}^0$ metrics via
  regularizing {R}icci flow.
\newblock {\em Geometric and Functional Analysis}, 29:1703 --1772, 2019.

\bibitem{PBG20}
Paula Burkhardt-Guim.
\newblock Defining pointwise lower scalar curvature bounds for ${C}^0$ metrics
  with regularization by {R}icci flow.
\newblock {\em Symmetry, Integrability and Geometry: Methods and Applications
  (SIGMA)}, 16:10 pages, 2020.

\bibitem{PBG23}
Paula Burkhardt-Guim.
\newblock {A}{D}{M} mass for ${C}^0$ metrics and distortion under
  {R}icci-{D}e{T}urck flow.
\newblock {\em Journal f\"ur die reine und angewandte Mathematik}, 2024, 2023.

\bibitem{CaiWang26}
Jing-Bin Cai and Bing Wang.
\newblock Tthe {R}icci-{D}e{T}urck flow on complete manifolds.
\newblock arXiv:2603.22834 [math.DG]. \url{https://arxiv.org/abs/2603.22834}.
\newblock Accessed June 25, 2026.

\bibitem{ChodoshMantoulidisSchulze26}
Otis Chodosh, Christos Mantoulidis, and Felix Schulze.
\newblock Generic regularity for minimizing hypersurfaces in dimensions 9 and
  10.
\newblock {\em Publications math\'ematiques de l'IHES}, page to appear, 2026.

\bibitem{Chrusciel88}
P.T. Chru\'sciel.
\newblock On the invariant mass conjecture in general relativity.
\newblock {\em Communications in Mathematical Physics}, 120:233 -- 248, 1988.

\bibitem{ChuLeeWan26}
Jianchun Chu, Man~Chun Lee, and Jingbo Wan.
\newblock Rigidity of positive mass theorem with fast metric decay.
\newblock arXiv:2607.17236 [math.DG]. \url{https://arxiv.org/abs/2607.17236}.
\newblock Accessed July 21, 2026.

\bibitem{ChuLeeZhu22}
Jianchun Chu, Man~Chun Lee, and Jintian Zhu.
\newblock Singular positive mass theorem with arbitrary ends.
\newblock arXiv:2210.08261 [math.DG]. \url{https://arxiv.org/abs/2210.08261}.
\newblock Accessed June 5, 2024.

\bibitem{DaiMa07}
Xianzhe Dai and Li~Ma.
\newblock Mass under the {R}icci flow.
\newblock {\em Communications in Mathematical Physics}, 274:65 -- 80, 2007.

\bibitem{DaiSunWang25}
Xianzhe Dai, Yukai Sun, and Changliang Wang.
\newblock The positive mass theorem for asymptotically flat manifolds with
  isolated conical singularities.
\newblock {\em Science China Mathematics}, 68:1671 -- 1686, 2025.

\bibitem{DeTurck83}
Dennis DeTurck.
\newblock Deforming metrics in the direction of their {R}icci tensors.
\newblock {\em Journal of Differential Geometry}, 18:157 -- 162, 1983.

\bibitem{FogagnoloGattiPluda26}
Mattia Fogagnolo, Giorgio Gatti, and Alessandra Pluda.
\newblock Scalar curvautre bounds for 3{D} continuous metrics through the
  {I}nverse {M}ean {C}urvature {F}low.
\newblock arXiv:2604.14087 [math.DG]. \url{https://arxiv.org/abs/2604.14087}.
\newblock Accessed June 23, 2026.

\bibitem{Friedman64}
Avner Friedman.
\newblock {\em Partial Differential Equations of Parabolic Type}.
\newblock Prentice--Hall, Englewood Cliffs, N.J., 1964.

\bibitem{Gromov21}
Misha Gromov.
\newblock Four lectures on scalar curvature.
\newblock arXiv:1908.10612v6 [math.DG]. \url{https://arxiv.org/abs/1908.10612}.
\newblock Accessed June 5, 2024.

\bibitem{Gromov14}
Misha Gromov.
\newblock Dirac and {P}lateau billiards in domains with corners.
\newblock {\em Central European Journal of Mathematics}, 12(8):1109--1156,
  2014.

\bibitem{Hafemann26}
Eduardo Hafemann.
\newblock A low-regularity {R}iemannian positive mass theorem for non-spin
  manifolds with distributional curvature.
\newblock arXiv:2602.03451 [math.DG]. \url{https://arxiv.org/abs/2602.03451}.
\newblock Accessed June 24, 2026.

\bibitem{Huisken06}
Gerhard Huisken.
\newblock {\em An isoperimetric concept for mass and quasilocal mass},
  volume~3, pages 87--88.
\newblock European Mathematical Society (EMS), Zurich, 2006.

\bibitem{Jauregui20-2}
Jeffrey~L. Jauregui.
\newblock Scalar curvature and the relative capacity of geodesic balls.
\newblock {\em Proceedings of the American Mathematical Society}, 149(11),
  2021.

\bibitem{JaureguiLeeUnger24}
Jeffrey~L. Jauregui, Dan~A. Lee, and Ryan Unger.
\newblock A note on {H}uisken's isoperimetric mass.
\newblock {\em Letters in Mathematical Physics}, 114(134), 2024.

\bibitem{JiangShengZhang22}
Wenshuai Jiang, Weimin Sheng, and Huaiyu Zhang.
\newblock Removable singularity of positive mass theorem with continuous
  metrics.
\newblock {\em Mathematische Zeitschrift}, 302:302, 2022.

\bibitem{KochLamm12}
Herbert Koch and Tobias Lamm.
\newblock Geometric flows with rough initial data.
\newblock {\em Asian Journal of Mathematics}, 16:209--236, 2012.

\bibitem{Lee19-errata}
Dan~A. Lee.
\newblock Errata for \textbf{Geometric Relativity} by dan a. lee, as of march
  8, 2024.
\newblock
  \url{https://www.ams.org/publications/authors/books/submit-book/gsm-201-errata.pdf}.
\newblock Accessed April 13, 2026.

\bibitem{Lee19}
Dan~A. Lee.
\newblock {\em Geometric Relativity}, volume 201 of {\em Graduate Studies in
  Mathematics}.
\newblock American Mathematical Society, Providence, Rhode Island, 2019.

\bibitem{LeeLeFloch15}
Dan~A. Lee and Philippe~G. LeFloch.
\newblock The positive mass theorem for manifolds with distributional
  curvature.
\newblock {\em Communications in Mathematical Physics}, 339:99--120, 2015.

\bibitem{LeeLitzingerSimon26}
Man-Chun Lee, Florian Litzinger, and Miles Simon.
\newblock Spaces with distributional scalar curvature bounded from below:
  Optimal regularity and positive mass.
\newblock arXiv:2606.23272 [math.DG]. \url{https://arxiv.org/abs/2606.23272}.
\newblock Accessed June 24, 2026.

\bibitem{Li20-2}
Chao Li.
\newblock Dihedral rigidity of parabolic polyhedrons in hyperbolic spaces.
\newblock {\em Symmetry, Integrability and Geometry: Methods and Applications
  (SIGMA)}, 219(099):8 pages, 2020.

\bibitem{Li18}
Yu~Li.
\newblock Ricci flow on asymptotically {E}uclidean manifolds.
\newblock {\em Geometry and Topology}, 22(23):1837--1891, 2018.

\bibitem{Lohkamp99}
Joachim Lohkamp.
\newblock Scalar curvature and hammocks.
\newblock {\em Mathematische Annalen}, 313:385 -- 407, 1999.

\bibitem{MazurowskiYao26-2}
Liam Mazurowski and Xuan Yao.
\newblock A {P}ositive {M}ass {T}heorem for continuous metrics.
\newblock arXiv:2606.19123 [math.DG]. \url{https://arxiv.org/abs/2606.19123}.
\newblock Accessed June 23, 2026.

\bibitem{MazurowskiYao26-1}
Liam Mazurowski and Xuan Yao.
\newblock Quantification of $c^0$ convergence in dimension three.
\newblock arXiv:2604.14087 [math.DG]. \url{https://arxiv.org/abs/2604.14087}.
\newblock Accessed June 23, 2026.

\bibitem{McFeronSzekelyhidi12}
Donovan McFeron and G\'abor Székelyhidi.
\newblock On the positive mass theorem for manifolds with corners.
\newblock {\em Communications in Mathematical Physics}, 313, 2012.

\bibitem{Miao02}
Pengzi Miao.
\newblock Positive mass theorem on manifolds admitting corners along a
  hypersurface.
\newblock {\em Advances in Theoretical and Mathematical Physics}, 6:1163 --
  1182, 2002.

\bibitem{BrendleWang26}
Y.~Wang S.~Brendle.
\newblock A dimension descent scheme for the positive mass theorem in arbitrary
  dimension.
\newblock arXiv:2604.08473 [math.DG]. \url{https://arxiv.org/abs/2604.08473}.
\newblock Accessed June 23, 2026.

\bibitem{SchoenYau79}
Richard Schoen and Shing-Tung Yau.
\newblock On the proof of the positive mass conjecture in general relativity.
\newblock {\em Communications in Mathematical Physics}, 65:45--76, 1979.

\bibitem{SchoenYau81}
Richard Schoen and Shing-Tung Yau.
\newblock Proof of the positive mass theorem. {I}{I}.
\newblock {\em Comm. Math. Phys.}, 79(2):231 -- 260, 1981.

\bibitem{SchoenYau19}
Richard Schoen and Shing-Tung Yau.
\newblock Positive scalar curvature and minimal hypersurface singularities.
\newblock {\em Surv. Diff. Geom.}, 24(1), April 2017.
\newblock DOI:10.4310/SDG.2019.v24.n1.a10.

\bibitem{Schoen89}
Richard~M. Schoen.
\newblock {\em Variational Theory for the Total Scalar Curvature Functional for
  {R}iemannian Metrics and Related Topics}, volume Lecture Notes in
  Mathematics, 1365.
\newblock Springer, Berlin, Heidelberg, 1989.

\bibitem{Simon02}
Miles Simon.
\newblock Deformation of ${C}^0$ {R}iemannian metrics in the direction of their
  {R}icci curvature.
\newblock {\em Communications in Analysis and Geometry}, 10:1033 -- 1074, 2002.

\bibitem{Witten81}
Edward Witten.
\newblock A new proof of the positive energy theorem.
\newblock {\em Communications in Mathematical Physics}, 80:381--402, 1981.

\end{thebibliography}
\end{document}